\documentclass[reqno]{amsart}
\usepackage{graphicx} % Required for inserting images
\usepackage{amsmath}
\usepackage{bm}
\usepackage{amsfonts}
\usepackage{amssymb}
\usepackage{color}
\usepackage{mathtools}
\usepackage{amsthm}
\usepackage{float}
\usepackage[authoryear,round]{natbib}
\usepackage{xcolor}
\usepackage[colorlinks=true,citecolor=blue]{hyperref}

\usepackage[
  a4paper,
  top=40mm,
  bottom=40mm,
  left=30mm,
  right=30mm
]{geometry}

\usepackage{setspace}
\newcommand{\eps}{\varepsilon}
\newcommand\R{\mathbb{R}}
\newcommand\bR{{\mathbb{R}}}
\newcommand\bC{{\mathbb{C}}}

\newcommand{\ovl}[1]{\overline{#1}}

\allowdisplaybreaks % 数式がページを跨いでもOK

\theoremstyle{plain}
\newtheorem{thm}{Theorem}            % 定理環境
\newtheorem{prop}[thm]{Proposition} % 命題（同連番）
\newtheorem{cor}[thm]{Corollary}     % 系（同連番）

\theoremstyle{definition}
\newtheorem{ass}[thm]{Assumption}    % 仮定（同連番）

\theoremstyle{remark}
\newtheorem{rem}[thm]{Remark}         % 注意・補足（同連番）

\begin{document}

\title[Pulse dynamics in a bulk--surface system]{Geometry--induced pulse dynamics in a bulk--surface reaction--diffusion system for cell polarization}

\author[R. Watanabe]{Riku Watanabe}
\address{Department of Mathematics, Faculty of Science, Hokkaido University, Hokkaido, 060-0810, Japan}
\email{watanabe.riku.y9@elms.hokudai.ac.jp}

\thanks{Date: \today. }

\begin{abstract}
This paper studies a bulk--surface reaction--diffusion system for cell polarization in two-dimensional domains.
The model describes the formation of localized patterns through the wave-pinning mechanism, while explicitly incorporating the effect of cell shape.
Using singular perturbation methods, we formally derive reduced ordinary differential equations describing the wave-pinning dynamics on a fast time scale and the subsequent slow drift of pulse solutions induced by domain geometry.
The resulting slow dynamics is a gradient flow of a potential function whose geometry-dependent part is expressed in terms of the Neumann Green's function.
Moreover, we give a probabilistic interpretation of the potential function using the excursion Poisson kernel.
We then analyze the reduced dynamics in several concrete geometries, including dumbbell-shaped domains and perforated disks.
In these examples, we characterize stationary pulse positions, their stability, and the bifurcation structures arising from changes in geometric parameters.
To evaluate the geometric terms appearing in the reduced dynamics, we use a conformal mapping method to compute the Neumann Green's function for these domains. Our analysis reveals geometry-induced phenomena such as nontrivial stationary pulse locations and both supercritical and subcritical pitchfork bifurcations.
Finally, we perform numerical simulations to support the theoretical predictions.
\end{abstract}

\maketitle
%%%%%%%%%%%%%%%%%%%%%%%%%%%%%%%%%
\setcounter{tocdepth}{3}

\section{Introduction}

Cell polarity plays a fundamental role in many biological processes, including cell migration, cell division, morphogenesis, and signal transduction. Typically, molecules that are initially distributed almost uniformly inside a cell become spatially biased in response to stimuli and form a localized active region on the cell membrane. Mathematical models based on reaction--diffusion equations have been widely used to describe such polarity formation  \cite{goehring2011,otsuji2007}.

A representative mathematical model for cell polarity is the wave-pinning model (WP), which is a mass-conserving bistable reaction--diffusion system describing the dynamics of Rho GTPases  \cite{mori2008}. Using this model, they proposed the wave-pinning mechanism, in which an activated region first propagates and then stops at a finite position due to substrate depletion. This mechanism differs from the classical Turing instability in that localization of the active protein is driven by bistable reaction kinetics together with mass conservation. Since then, the WP model has served as a simple and fundamental framework for explaining the formation and maintenance of localized active regions in cell polarity. For mathematical studies of the WP model, see  \cite{gomez2023,ikeda2025,mori2011,kuwamura2024,mori2022,yotsutani2015}.

In the standard WP model, however, the active and inactive forms are often treated as variables defined on the same spatial domain. In actual cell polarity, the active molecules mainly diffuse on the cell membrane, whereas the inactive molecules diffuse in the cytosol. Therefore, the shape of the cell membrane may directly affect pattern formation and the selection of localized regions. A natural framework for describing such geometric effects is a bulk--surface reaction--diffusion system.

In a bulk--surface system, the concentrations on the cell membrane, or surface, and in the cytosol, or bulk, are described by an unknown function $u$ on the boundary and an unknown function $v$ in the interior domain, respectively. In this paper, we consider the following model, which has been used in studies of bulk--surface cell-polarity models  \cite{giese2015,ramirez2015}:
\begin{equation}\label{eq:BS}\tag{BS}
\left\{
\begin{aligned}
\eps u_t &= \eps^2 \Delta_\Gamma u + f(u,v)\quad &&\text{on }\Gamma,\\
\eps v_t &= D\Delta v \quad &&\text{in }\Omega,\\
D\partial_n v &= -f(u,v)\quad &&\text{on }\Gamma,
\end{aligned}
\right.
\end{equation}
where $\Omega\subset\mathbb{R}^2$ or $\mathbb{R}^3$ is a bounded domain and $\Gamma=\partial\Omega$ is its boundary. Moreover, $0<\eps\ll1$, $\partial_n$ denotes the outward normal derivative, and $D>0$ is a constant. The precise assumptions on this model are stated in Section~\ref{sec:preliminary}. The boundary condition represents the flux of the cytosolic concentration induced by reactions on the membrane. As a consequence, \eqref{eq:BS} has the following conservation law:
\begin{align}\label{eq:mass-conservation}
    M:=\int_\Gamma u(0,y)\,dS_y+\int_\Omega v(0,x)\,dx
    =\int_\Gamma u(t,y)\,dS_y+\int_\Omega v(t,x)\,dx
    \quad (t>0).
\end{align}
Thus, as in the WP model, the system \eqref{eq:BS} exhibits the wave-pinning phenomenon and forms localized pulse solutions. However, since $v$ diffuses in $\Omega$ and is coupled to $u$ through the boundary flux, the domain geometry enters the dynamics through the bulk-mediated nonlocal interaction on the boundary. This is an essential difference from the standard WP model formulated on a single spatial domain. For a detailed discussion of the relationship between the WP model, \eqref{eq:BS}, and related models, see \cite{morita2021}.

Because of this feature, the influence of domain shape on \eqref{eq:BS} has attracted considerable interest. In particular, numerical studies have investigated where localized pulses on the boundary move under the influence of domain geometry  \cite{cusseddu2022,giese2015,ramirez2015}. We focus on the dynamics of pulse solutions of the two-dimensional version of \eqref{eq:BS}. As shown in Figure~\ref{fig:sim}, a pulse solution slowly moves due to the effect of the domain geometry and eventually stops. Such slow dynamics induced by geometry has been extensively studied in the field of pattern formation  \cite{chen2011,ishii2026,ni1993,sakajo2021}.

\begin{figure}[t]
\centering
\includegraphics[width=0.85\textwidth]{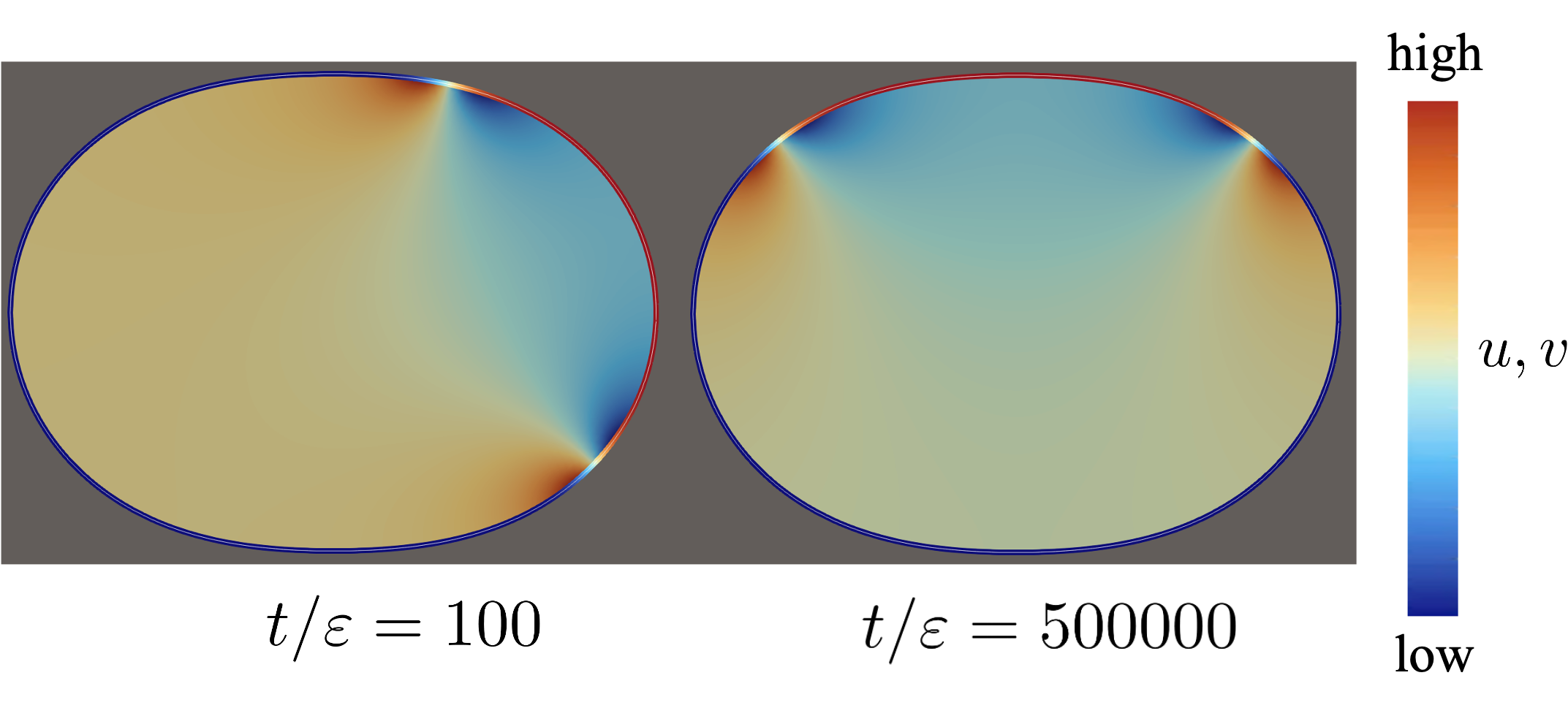}
\caption{Numerical simulation of the slow drift of a pulse solution in \eqref{eq:model}. The pulse moves along the boundary under the effect of domain geometry and approaches a stationary location.}
\label{fig:sim}
\end{figure}

The system \eqref{eq:BS} and related models have been studied as models for signal transduction, GTPase cycles, membrane-protein clustering, and receptor--ligand binding. In particular, when the domain is radially symmetric, classical Turing instabilities and non-classical pattern formation induced by bulk coupling have been investigated  \cite{borgqvist2021,morita2020,ratz2014}. The limit of large cytosolic diffusivity $(D\to\infty)$ has also been studied. In this limit, the bulk variable becomes almost spatially homogeneous, and the system is reduced to a nonlocal scalar reaction--diffusion equation on the boundary, called a shadow system  \cite{hausberg2018}. Such shadow systems have been used as an effective tool for reducing the analysis of \eqref{eq:BS} and related models to lower-dimensional nonlocal problems. In particular, the existence of stationary patterns and the derivation of interface equations have been studied for shadow systems  \cite{miller2023,morita2021}. However, numerical simulations indicate that the shadow system and the original system \eqref{eq:BS} do not necessarily exhibit the same behavior; in particular, the spatial distribution of $v$ can play an essential role  \cite{cusseddu2022}.

Thus, much of the existing analytical work on \eqref{eq:BS} has focused on the instability of spatially homogeneous steady states in radially symmetric domains or on the shadow limit. On the other hand, these approaches simplify the influence of geometry and therefore do not fully describe how the domain shape affects localized patterns. One reason why such effects remain analytically challenging is that \eqref{eq:BS} couples unknown functions defined on spaces of different dimensions, so that analytical methods developed for standard reaction--diffusion systems cannot be applied directly. In this paper, we extend singular perturbation methods so that they can be applied to \eqref{eq:BS}, and we investigate how the dynamics of pulse solutions is affected by the domain geometry.

More precisely, we formally derive the following ODEs for the half-width $w(t)$ and the center location $s_0(t)$ of a pulse solution on two different time scales:

\noindent\textbf{Main result.}
\begin{align}\label{eq:ODE1}
t=O(1):\left\{
\begin{aligned}
\frac{dw}{dt}&=\frac{J(v_0(w))}{\kappa(v_0(w))},\\
\frac{ds_0}{dt}&=0
\end{aligned}
\right. \qquad \text{(wave-pinning)}
\end{align}
and
\begin{align}\label{eq:ODE2}
t=O(\eps^{-2}):
\left\{
\begin{aligned}
w&=w_*,\\
\frac{ds_0}{dt}
&=-\eps^2\frac{\Delta h(v_*)J'(v_*)}{4\kappa_*D}E'(s_0;w_*)
\end{aligned}
\right. \qquad \text{(  slow dynamics)}
\end{align}
at leading order. 

Here, $J(v)$ is the function introduced in Assumption~\ref{ass:f}--(2), and $v_*$ denotes the root of $J$ specified in the same assumption. Moreover, $w_*$ is the half-width of the stationary pulse solution defined in \eqref{eq:def-w*}. The functions $v_0$ and $\kappa$ are defined in Subsection~\ref{subsec:ODE-w}, and they satisfy $v_0(w_*)=v_*$, $v_0'(w_*)<0$, and $\kappa>0$. The function $E$ encodes the geometric information of $\Omega$; more precisely, $E$ is defined in \eqref{eq:E-def} in terms of the Neumann Green's function of $\Omega$. Finally, we set $\kappa_*:=\kappa(v_*)>0$, while $\Delta h(v_*)>0$ is defined as in Assumption~\ref{ass:f}.

Equation~\eqref{eq:ODE1} describes the wave-pinning mechanism on the fast time scale, during which the pulse location remains fixed. In this regime, a pulse solution with half-width $w_*$ is formed. Moreover, this equation shows that stable wave-pinning occurs only when $J'(v_*)>0$. This observation is consistent with the stability condition for stationary pulse solutions in the WP model  \cite{ikeda2025}.

Equation~\eqref{eq:ODE2} describes the slow drift of the pulse location while the pulse width is fixed at $w_*$; see Figure~\ref{fig:sim}. Under the condition $J'(v_*)>0$, this equation shows that the pulse location follows the gradient flow of the potential function $E(\cdot;w_*)$. Thus, \eqref{eq:ODE2} explicitly describes the effect of the domain geometry on the pulse dynamics.

Furthermore, based on the reduced ODE \eqref{eq:ODE2}, we investigate pulse dynamics in concrete domains, such as a dumbbell-shaped domain and a perforated disk, under the condition $J'(v_*)>0$. More precisely, we compute the potential $E$ for these domains and analyze its critical points in order to determine the locations and stability of stationary pulse solutions. Through these calculations, we analyze the bifurcation structure of stationary points and reveal geometry-induced phenomena, including nontrivial equilibria.

Next, we provide a probabilistic interpretation of the potential function $E$ derived above. 
Here we assume that $\Omega$ is a simply connected domain with smooth boundary. 
For this purpose, we introduce the excursion Poisson kernel. 
We define the usual Poisson kernel $P_\Omega$ as the kernel representing a harmonic function $u$ with smooth boundary data $\phi$ in the form
$u(x)=\int_{\partial\Omega}P_\Omega(x,y)\phi(y)\,ds_y$.
It is well known that the Poisson kernel admits a probabilistic interpretation: 
$P_\Omega(x,y)$ is the probability density that a Brownian particle starting from a point $x\in\Omega$ first hits the boundary near $y\in\partial\Omega$. 
For the relation between Brownian motion and harmonic functions, see \cite[Chapter~7]{legall2016}.

It is known that the following limit exists:
\begin{align}\label{eq:excursion-def}
P_{\partial \Omega}(x_1,x_2)
:=
\lim_{\delta\downarrow 0}
\delta^{-1}P_\Omega(x_1-\delta n_{x_1}, x_2),
\qquad
(x_1,x_2\in\partial\Omega,\ x_1\neq x_2).
\end{align}
The function $P_{\partial\Omega}(x_1,x_2)$ is called the excursion Poisson kernel. For distinct $x_1,x_2\in\partial\Omega$, $P_{\partial\Omega}(x_1,x_2)=P_{\partial\Omega}(x_2,x_1)$ holds.
For basic properties of the excursion Poisson kernel noted here, see \cite{kozdron2005}.
We then have the following corollary from Proposition~\ref{prop:Conformal-rep-of-E} in Section~\ref{sec:examples}.
\begin{cor}\label{cor:Poisson}
Let
\[
x_1=\gamma(s_0-w_*),
\qquad
x_2=\gamma(s_0+w_*).
\]
Then
\begin{align}\label{eq:E-Poisson}
E(s_0;w_*)
=
\pi^{-1}\log\!\left(\pi P_{\partial\Omega}(x_1,x_2)\right)
\end{align}
holds.
Consequently, by \eqref{eq:ODE2},
\begin{align}\label{eq:probablistic-ODE}
\frac{ds_0}{dt}
=
-\eps^2
\frac{\Delta h(v_*)J'(v_*)}
{4\pi\kappa_*D}
\frac{d}{ds_0}
\log P_{\partial \Omega}(x_1,x_2)
\end{align}
holds.
\end{cor}
The proof of this theorem is given in Appendix~\ref{app:poisson}.

By the definition of the excursion Poisson kernel, for fixed distinct points
$x_1,x_2\in\partial\Omega$,
\begin{align}
P_\Omega(x_1-\delta n_{x_1},x_2)
=
\delta P_{\partial\Omega}(x_1,x_2)
+
o(\delta)
\qquad
\text{as }\delta\to0.
\end{align}
This reflects the fact that a Brownian particle starting just inside the boundary near $x_1$ hits the boundary again near $x_1$ in most cases. 
On the other hand, although the corresponding hitting density is small, some particles reach a neighborhood of a boundary point $x_2$. 
This hitting density is of order $O(\delta)$, and the excursion Poisson kernel gives the coefficient of its leading-order term.

When $J'(v_*)>0$, \eqref{eq:probablistic-ODE} shows that the pulse location follows the gradient flow of
$\log P_{\partial\Omega}(\gamma(\cdot-w_*),\gamma(\cdot+w_*))$.
Thus, formally speaking, \eqref{eq:probablistic-ODE} implies that a pulse of fixed width moves so as to make it less likely for a Brownian particle, starting near one transition layer and absorbed upon hitting the boundary, to reach a neighborhood of the other transition layer.

Figure~\ref{fig:poisson} illustrates this probabilistic interpretation by adding arrows to the numerical simulation in the dumbbell-shaped domain.
At the initial stage shown in the left panel, a Brownian particle starting near one transition layer can relatively easily reach the neighborhood of the other transition layer through the bulk.
In contrast, at the stationary state shown in the right panel, the narrow neck of the dumbbell-shaped domain hinders such Brownian excursions, making it more difficult for a particle starting near one transition layer to reach the neighborhood of the other transition layer.
This interpretation makes it possible to predict the pulse dynamics intuitively by considering the motion of Brownian particles, without conducting numerical simulations.

\begin{figure}[t]
\centering
\includegraphics[width=0.95\textwidth]{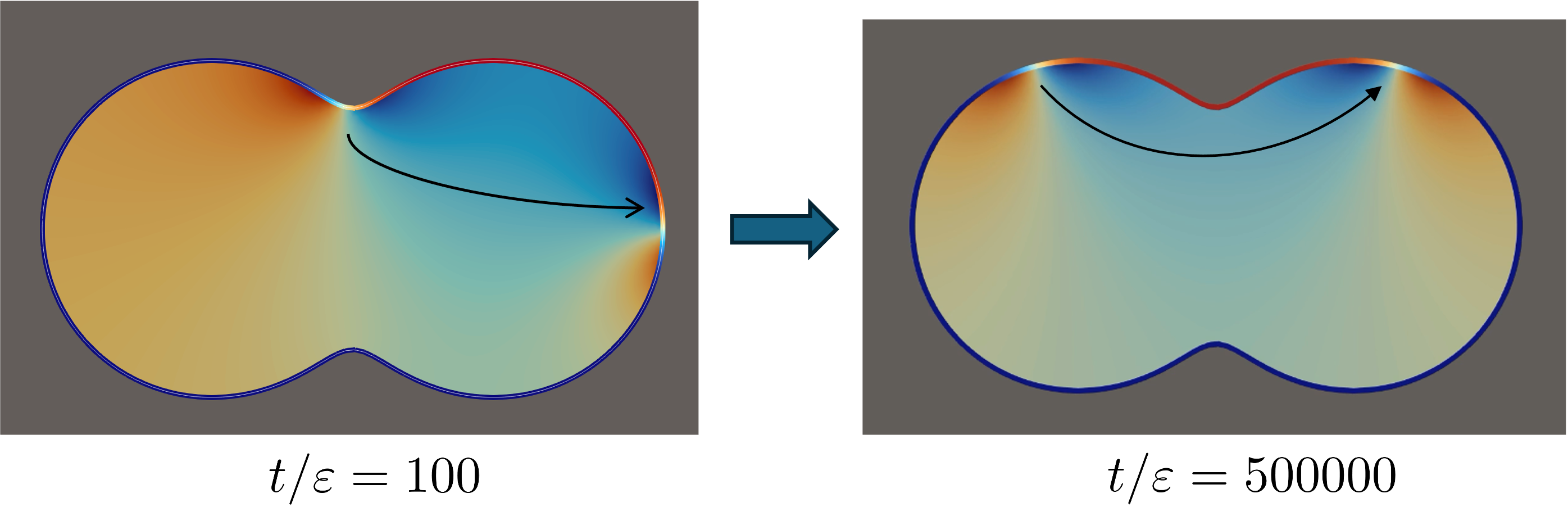}
\caption{
The numerical simulation is performed using the same parameter values as in Fig.~\ref{fig:sim-dumbbell}.
The initial condition is the same as that used in the upper simulation in Fig.~\ref{fig:sim-dumbbell}.
}
\label{fig:poisson}
\end{figure}

The rest of this paper is organized as follows. In Section~\ref{sec:preliminary}, we state the problem setting. In Section~\ref{sec:ODE}, we derive the ODEs \eqref{eq:ODE1} and \eqref{eq:ODE2}. In Section~\ref{sec:examples}, we compute the potential $E$ for concrete domains and analyze the corresponding bifurcation structures. Finally, Section~\ref{sec:conclusion} is devoted to concluding remarks.

\section{Problem setting}\label{sec:preliminary}

In this paper, we consider a bounded domain $\Omega$ in $\bR^2$. In this case, \eqref{eq:BS} can be written as
\begin{equation}\label{eq:model}
\left\{
\begin{aligned}
\eps u_t &= \eps^2  u_{ss} + f(u,v)\quad &&\text{on }\Gamma,\\
\eps v_t &= D\Delta v \quad &&\text{in }\Omega,\\
D\partial_n v &= -f(u,v)\quad &&\text{on }\Gamma.
\end{aligned}
\right.
\end{equation}
Here, $\Gamma:=\partial\Omega$ is a closed $C^4$ curve, $L:=|\Gamma|$ denotes its length, and
$s\in[0,L)$ denotes the arc-length parameter on $\Gamma$. We write
$\gamma:[0,L)\to\bR^2$ for an arc-length parametrization of $\Gamma$.
The unknown functions are $u=u(s,t)$ and $v=v(x,t)$.
We assume that $f$ satisfies the following conditions.

\begin{ass}\label{ass:f}
\begin{enumerate}
\item \textup{(Bistability)}
The function $f(u,v)$ is smooth, namely $f\in C^\infty(\bR^2)$.
There exist constants $v_\mathrm{min}$ and $v_\mathrm{max}$ such that, for each fixed
\(v\in (v_\mathrm{min},v_\mathrm{max})\), the equation \(f(u,v)=0\) has exactly three roots
\[
h^{-}(v)<h^{0}(v)<h^{+}(v).
\]
Moreover, these roots satisfy
\[
f_{u}\!\bigl(h^{\pm}(v),v\bigr)<0,\qquad
f_{u}\!\bigl(h^{0}(v),v\bigr)>0.
\]
We define $\Delta h(v):=h^+(v)-h^-(v)$.

\item \textup{(Mass balance condition)}
Define
\[
J(v):=\int_{h^{-}(v)}^{h^{+}(v)} f(u,v)\,du.
\]
We assume that there exists an isolated root $v=v_*\in (v_\mathrm{min},v_\mathrm{max})$ of $J(v)=0$ satisfying
\[
J'(v_*)\neq 0.
\]

\item \textup{(Stability of homogeneous states )}
For any \(v\in (v_\mathrm{min},v_\mathrm{max})\), we assume that
\[
f_{u}\!\bigl(h^{\pm}(v),v\bigr)<0< f_{v}\!\bigl(h^{\pm}(v),v\bigr).
\]
Under this assumption, \cite{mori2011} observed that the homogeneous states $(u,v)=(h^\pm(v),v)$ are stable for $v\in (v_\mathrm{min},v_\mathrm{max})$ in the WP model.
\end{enumerate}
\end{ass}

For example, the function
\[
f(u,v)=\Bigl(k_0+\gamma_0\frac{u^2}{1+u^2}\Bigr)v-u,
\qquad \gamma_0>8k_0>0,
\]
satisfies the above assumptions.
We also assume that the total mass $M$ lies in the admissible range
\begin{align}\label{ass:M}
    M\in\bigl(v_*|\Omega|+Lh^-(v_*),\ v_*|\Omega|+Lh^+(v_*)\bigr).
\end{align}
We define the Neumann Green's function $G$ as the unique solution of
\begin{align*}
\begin{cases}
\Delta_x G=\dfrac{1}{|\Omega|}-\delta(x-y) & \text{in }\Omega,\\
\partial_{n_x}G(x,y)=0 & \text{on }\Gamma,\\
\displaystyle \int_\Omega G(x,y)\,dx=0.
\end{cases}
\end{align*}
We denote its boundary trace by
\[
G_{\Gamma}(s,s'):=G(\gamma(s),\gamma(s')).
\]
Then, for \(s,s'\in[0,L)\) with \(s\neq s'\), we decompose \(G_\Gamma\) into its singular and regular parts as
\begin{align}\label{eq:G-decomp}
    G_{\Gamma}(s,s')
    =-\frac{1}{\pi}\log|\gamma(s)-\gamma(s')|+H(s,s').
\end{align}
The function \(G\) satisfies the following properties:
\begin{align}\label{eq:G-property}
\begin{aligned}
    G(x,y)&=G(y,x)
    && \text{for }x,y\in\Omega,\ x\neq y,\\
    G(\cdot,y)&\in C^4(\overline\Omega\setminus\{y\})
    && \text{for }y\in\Omega,\\
    G_{\Gamma}(s,s')&=G_\Gamma(s',s)
    && \text{for }s,s'\in[0,L),\ s\neq s',\\
    G_{\Gamma}(\cdot,s')&\in C^4_\mathrm{per}\bigl([0,L)\setminus\{s'\}\bigr)
    &&\text{for }s'\in[0,L).
\end{aligned}
\end{align}
In this paper, the subscript $\cdot_\mathrm{per}$ indicates periodicity with respect to the arc-length variable \(s\).
Moreover, \(H\) extends as a \(C^2\)-function to the whole set \([0,L)\times[0,L)\), and satisfies
\begin{align*}
H(s,s')&=H(s',s)
\quad \text{for }s,s'\in[0,L),\\
H&\in C^2_\mathrm{per}\bigl([0,L)\times[0,L)\bigr).
\end{align*}

\section{Formal derivation of interface ODEs}\label{sec:ODE}

In this section, we formally derive reduced equations of motion for pulse solutions of
\eqref{eq:model} with two sharp interfaces
\[
s=s_1(t),\qquad s=s_2(t),\qquad 0<s_1(t)<s_2(t)<L,
\]
for sufficiently small $\eps$.
We define the pulse center and the half-width by
\[
s_0(t):=\frac{s_1(t)+s_2(t)}{2},\qquad
w(t):=\frac{s_2(t)-s_1(t)}{2}.
\]
Without loss of generality, we may assume
\[
0<w<L/2,\qquad w<s_0<L-w.
\]

Suppose that \eqref{eq:model} is initially close to a homogeneous equilibrium state
\((u,v)=(h^\pm(v),v)\). By Assumption~\ref{ass:f}--(3), such homogeneous states are linearly stable, and hence sufficiently small perturbations return to a homogeneous equilibrium. On the other hand, if the initial perturbation is sufficiently large, the following dynamics is expected to occur on different time scales:
\begin{enumerate}
\item[(i)] sharp interfaces are formed on the time scale \(O(\eps|\log\eps|)\)
\cite{alfaro2008,chen1992};
\item[(ii)] the interfaces propagate on the \(O(1)\) time scale and stop when \(v\) reaches the value \(v=v_*\) satisfying the mass-balance condition in Assumption~\ref{ass:f}; this is the wave-pinning process;
\item[(iii)] the resulting pulse drifts slowly due to geometric effects on the \(O(\eps^{-2})\) time scale.
\end{enumerate}
We are interested in the processes (ii) and (iii). During these processes, the variable \(v\) relaxes on the time scale \(O(\eps D^{-1})=O(\eps)\). This relaxation is much faster than the \(O(1)\) and \(O(\eps^{-2})\) time scales considered here. Therefore, we adopt the quasi-stationary approximation
\[
\partial_t\tilde v=0,
\]
where
\[
\ovl v(t):=\frac1{|\Omega|}\int_\Omega v(x,t)\,dx,
\qquad
\tilde v(x,t):=v(x,t)-\ovl v(t).
\]
Thus, we consider the following system:
\begin{equation}\label{eq:model-quasi-stationary}
\left\{
\begin{aligned}
\eps u_t &= \eps^2 u_{ss} + f(u,v)\quad &&\text{on }\Gamma,\\
\eps \ovl v'(t) &= D\Delta  v \quad &&\text{in }\Omega,\\
D\partial_n v &= -f(u,v)\quad &&\text{on }\Gamma.
\end{aligned}
\right.
\end{equation}

\subsection{ODE for wave-pinning dynamics}\label{subsec:ODE-w}

In this subsection, we formally derive an ODE for \(w\) corresponding to the process (ii).
For a function \(\phi(t)\) depending only on \(t\), we write
\[
\dot\phi:=\frac{d\phi}{dt}.
\]
Since we focus on interfaces moving on the \(O(1)\) time scale, we assume
\[
\dot s_1,\dot s_2=O(1).
\]

We first construct the outer solution. Substituting the expansions
\[
u(s,t)=u_0(s,t)+O(\eps),\qquad
v(x,t)=v_0(x,t)+O(\eps)
\]
into \eqref{eq:model-quasi-stationary}, we obtain at \(O(1)\)
\begin{align*}
    \left\{
    \begin{aligned}
    0 &= f(u_0,v_0)\quad &&\text{on }\Gamma,\\
    0 &= \Delta v_0 \quad &&\text{in }\Omega,\\
    D\partial_n v_0 &= -f(u_0,v_0)=0\quad &&\text{on }\Gamma.
    \end{aligned}
    \right.
\end{align*}
The equation for \(v_0\) implies that \(v_0\) is spatially constant. In particular,
\(\tilde v(x,t)=O(\eps)\). The first equation then gives
\[
u_0=h^\pm(v_0(t)),\quad h^0(v_0(t)).
\]
Since we consider a pulse solution, we choose
\begin{align*}
    u_0(s,t)=
    h^-(v_0(t))+\Delta h(v_0(t))\,\chi_{(s_1,s_2)}(s),
\end{align*}
where \(\chi_{(s_1,s_2)}\) denotes the characteristic function of the interval \((s_1,s_2)\).
Hence,
\begin{align}\label{eq:u,v1}
\begin{aligned}
    u(s,t)&=
    h^-(v_0(t))+\Delta h(v_0(t))\,\chi_{(s_1,s_2)}(s)+O(\eps),\\
    v(x,t)&=v_0(t)+O(\eps).
\end{aligned}
\end{align}

Next, we construct the inner solution. Near the left interface \(s=s_1(t)\), we introduce
\[
\xi=\frac{s-s_1(t)}{\eps}
\]
and set
\[
u(s,t)= U(\xi;v_0(t))+O(\eps).
\]
At the left interface, \(u\) transitions from \(h^-(v_0(t))\) to \(h^+(v_0(t))\). Substitution into \eqref{eq:model} gives
\[
-\frac{\eps\dot s_1}{\eps}U_\xi
=
U_{\xi\xi}+f(U,v_0(t))+O(\eps),
\qquad
U(\pm\infty)=h^\pm(v_0(t)).
\]
Multiplying this equation by \(U_\xi\) and integrating over \(\xi\in\mathbb R\), we obtain
\[
-\dot s_1\int_{-\infty}^{\infty}(U_\xi)^2\,d\xi
=
\int_{h^-(v_0(t))}^{h^+(v_0(t))} f(u,v_0(t))\,du+O(\eps)
=
J(v_0(t))+O(\eps).
\]
Therefore,
\begin{equation*}
\dot s_1(t)=-\frac{J(v_0(t))}{\kappa(v_0(t))}+O(\eps),
\qquad
\kappa(v_0):=\int_{-\infty}^{\infty}(U_\xi(\xi;v_0))^2\,d\xi .
\end{equation*}
At \(s=s_2(t)\), the function \(u\) transitions from \(h^+(v_0)\) to \(h^-(v_0)\). The same calculation gives
\begin{equation*}
\dot s_2(t)=\frac{J(v_0(t))}{\kappa(v_0(t))}+O(\eps).
\end{equation*}
Consequently,
\begin{align*}
    \dot w(t)
    =\frac{\dot s_2(t)-\dot s_1(t)}{2}
    =\frac{J(v_0(t))}{\kappa(v_0(t))}+O(\eps),
    \qquad
    \dot s_0(t)
    =\frac{\dot s_2(t)+\dot s_1(t)}{2}
    =O(\eps).
\end{align*}
Thus, on this time scale, the pulse center remains almost fixed, while the pulse width evolves. Substituting \eqref{eq:u,v1} into the conservation law \eqref{eq:mass-conservation}, we obtain
\begin{align}\label{eq:F=0}
    F(v_0(t),w(t))
    :=2w(t)h^+(v_0(t))+(L-2w(t))h^-(v_0(t))+|\Omega|v_0(t)-M=0,
    \qquad t>0.
\end{align}
Differentiating \(f(h^\pm(v),v)=0\), which follows from Assumption~\ref{ass:f}--(1), and using Assumption~\ref{ass:f}--(3), we obtain
\[
(h^\pm)'(v)
=
-\frac{f_v(h^\pm(v),v)}{f_u(h^\pm(v),v)}
>0
\]
for \(v\in(v_\mathrm{min},v_\mathrm{max})\). Hence,
\[
\partial_{v_0}F(v_0,w)
=
2w(h^+)'(v_0)+(L-2w)(h^-)'(v_0)+|\Omega|
>0
\]
for all \(t>0\). Therefore, by the implicit function theorem, \eqref{eq:F=0} can be solved as
\[
v_0=v_0(w).
\]
Thus, we obtain
\begin{align}\label{eq:w-dot1}
\dot w=\frac{J(v_0(w))}{\kappa(v_0(w))}
\end{align}
at leading order. We define
\begin{align}\label{eq:def-w*}
w_*:=\frac{M-v_*|\Omega|-Lh^-(v_*)}{2\Delta h(v_*)}.
\end{align}
By the admissibility condition \eqref{ass:M}, we have \(0<w_*<L/2\). 
Hence \(w_*\) satisfies $F(v_*,w_*)=0$ or equivalently $v_0(w_*)=v_*$.
Together with Assumption~\ref{ass:f}--(2), this implies that \(w=w_*\) is an equilibrium of \eqref{eq:w-dot1}. A direct calculation shows that the linearized eigenvalue at this equilibrium is
\[
\frac{J'(v_*)v_0'(w_*)}{\kappa(v_*)}.
\]
Differentiating \eqref{eq:F=0} with respect to \(w\), we obtain
\[
v_0'(w)
=
\frac{-2\Delta h(v_0)}
{2w(h^+)'(v_0)+(L-2w)(h^-)'(v_0)+|\Omega|}
<0.
\]
Therefore, if \(J'(v_*)>0\), then \(w=w_*\) is a stable equilibrium of \eqref{eq:w-dot1}. Conversely, if \(J'(v_*)<0\), then \(w=w_*\) is unstable, suggesting that the corresponding pulse solution is unstable. This linearized eigenvalue corresponds to the principal eigenvalue associated with the linearization around stationary pulse solutions in the WP model \cite{ikeda2025}.

\subsection{ODE for slow dynamics}\label{subsec:ODE-s0}

In this subsection, we formally derive an equation of motion for the pulse center \(s_0\) corresponding to the process (iii). We consider the case where wave-pinning occurs stably, and therefore assume $J'(v_*)>0$.
As described above, the geometric drift is expected to occur on the \(O(\eps^{-2})\) time scale. We introduce the slow time
\[
T=\eps^2 t
\]
and assume
\[
\dot s_1,\dot s_2=O(1),
\]
where, in this subsection, for a function \(\phi(T)\) depending only on \(T\), we write
\[
\dot\phi:=\frac{d\phi}{dT}.
\]
For simplicity of notation, we use \(T\) instead of \(t\) as the argument of
$u,v,\ovl v,\tilde v,s_1,s_2,s_0,w$.
Since \(\partial_t=\eps^2\partial_T\), the quasi-stationary system \eqref{eq:model-quasi-stationary} becomes
\begin{equation}\label{eq:model-GD}
\left\{
\begin{aligned}
\eps^3 u_T &= \eps^2 u_{ss} + f(u,v)\quad &&\text{on }\Gamma,\\
\eps^3\ovl v' &= 
D\Delta v \quad &&\text{in }\Omega,\\
D\partial_n v &= -f(u,v)\quad &&\text{on }\Gamma.
\end{aligned}
\right.
\end{equation}

We first consider the outer solution. We expand
\[
u(s,T)=u_0(s,T)+O(\eps^3),
\qquad
v(x,T)=v_0(x,T)+O(\eps^3).
\]
The \(O(1)\) terms in \eqref{eq:model-GD} are
\[
0=f(u_0,v_0)\quad\text{on }\Gamma,\qquad
0=\Delta v_0\quad\text{in }\Omega,\qquad
\partial_n v_0=0\quad\text{on }\Gamma.
\]
Thus, \(v_0\) is spatially constant, and we write
\[
v_0=v_0(T).
\]
Consequently,
\begin{equation}\label{eq:gd-u0}
u_0(s,T)=h^-(v_0(T))+\Delta h(v_0(T))\,\chi_{(s_1,s_2)}(s).
\end{equation}
Expanding
\[
\ovl v(T)=\ovl v_0(T)+\eps \ovl v_1(T)+\eps^2 \ovl v_2(T)+O(\eps^3),
\]
we obtain
\begin{align}\label{eq:ovlv0=v0}
    \ovl v_0(T)=v_0(T).
\end{align}

\begin{rem}
Strictly speaking, the outer expansion of \(u\) should be written as
$u=u_0+\eps u_1+\eps^2u_2$.
However, if one carries out the calculation with this expansion, the terms \(u_1\) and \(u_2\) do not contribute to \(\dot s_0\) and hence do not affect the determination of the pulse location. We therefore omit them.
\end{rem}

Next, we consider the inner solution. Near \(s=s_1(T)\), we introduce
\[
\xi=\frac{s-s_1(T)}{\eps}
\]
and expand
\begin{align*}
u&=U_0(\xi,T)+\eps U_1(\xi,T)+ \eps^2 U_2(\xi,T) +O(\eps^3),\\
v&=V_0(x,T)+\eps V_1(x,T)+ \eps^2 V_2(x,T) +O(\eps^3).
\end{align*}
Then
\[
u_T=\partial_TU_{0}-\frac{\dot s_{1}}{\eps}\partial_\xi U_0
+\eps\partial_TU_1-\dot s_{1}\partial_\xi U_1
+\eps^2\partial_TU_2-\eps \dot s_{1}\partial_\xi U_2,
\]
and
\[
u_{ss}=\frac{1}{\eps^2}\partial_\xi^2U_{0}
+\frac1\eps\partial_\xi^2U_1+\partial_\xi^2U_2.
\]
Substituting these into the first equation of \eqref{eq:model-GD} and comparing terms of each order, we obtain
\begin{align}\label{eq:inner-expansion}
\begin{cases}
\partial_\xi^2U_0+f(U_0,V_0)=0,\\
\mathcal{L}U_1=-f_v(U_0,V_0)V_1, \\
-\dot s_1\partial_\xi U_0=\mathcal{L}U_2+f_v(U_0,V_0)V_2+f_*,
\end{cases}
\end{align}
where
\begin{align*}
\mathcal{L}&:=\partial_\xi^2+f_u(U_0,V_0),\\
f_*&:=\frac12f_{uu}(U_0,V_0)U_1^2
+f_{uv}(U_0,V_0)U_1V_1
+\frac12f_{vv}(U_0,V_0)V_1^2.
\end{align*}
Similarly, near \(s=s_2(T)\), by setting
\[
\xi=\frac{s_2(T)-s}{\eps},
\]
we obtain
\begin{align}\label{eq:inner-expansion2}
\begin{cases}
\partial_\xi^2U_0+f(U_0,V_0)=0,\\
\mathcal{L}U_1=-f_v(U_0,V_0)V_1, \\
\dot s_2\partial_\xi U_0=\mathcal{L}U_2+f_v(U_0,V_0)V_2+f_*.
\end{cases}
\end{align}
For simplicity, we use the same notation
$U_0,U_1,U_2,V_0,V_1,V_2,f_*$
near \(s=s_1\) and \(s=s_2\). When a distinction is needed, we use the superscript \(\cdot^{(i)}\).

We next compute \(v\). From the second and third equations of \eqref{eq:model-GD}, we have
\[
v(x,T)
=\ovl v(T)+
\int_\Gamma G\bigl(x,\gamma(s')\bigr)\,\partial_n v(\gamma(s'),T)\,ds'.
\]
On the other hand, \eqref{eq:model-GD} gives
\[
\partial_n v=-\frac1D f(u,v)
=\frac1D(\eps^2u_{ss}-\eps^3u_T).
\]
Therefore,
\begin{equation}\label{eq:v-green}
v(x,T)
=
\ovl v(T)
-\frac{\eps^3}{D}\int_\Gamma G\bigl(x,\gamma(s')\bigr)u_T(s',T)\,ds'
+\frac{\eps^2}{D}\int_\Gamma G\bigl(x,\gamma(s')\bigr)u_{ss}(s',T)\,ds'.
\end{equation}
Fix \(1/2<\alpha<1\), and set
\[
\Gamma_j(T):=[\,s_j(T)-\eps^\alpha,\ s_j(T)+\eps^\alpha\,]\qquad (j=1,2),
\qquad
\Gamma_{\mathrm{out}}(T):=\Gamma\setminus(\Gamma_1(T)\cup \Gamma_2(T)).
\]
From \eqref{eq:gd-u0}, we use the following approximate form of \(u\):
\begin{align}\label{eq:u-app}
    u(s,T) = 
    \begin{cases}
        h^-(v_*)+\Delta h(v_*)\,\chi_{(s_1,s_2)}(s)
        & \quad (s\in\Gamma_\mathrm{out}(T)),\\
        U_0\bigl((s-s_1(T))/\eps\bigr)
        & \quad (s\in\Gamma_1(T)),\\
        U_0\bigl((s_2(T)-s)/\eps\bigr)
        & \quad (s\in\Gamma_2(T)).
    \end{cases}
\end{align}
Since \(u_T=u_{ss}=0\) on \(\Gamma_{\mathrm{out}}\), only the contributions from \(\Gamma_1(T)\) and \(\Gamma_2(T)\) need to be considered in \eqref{eq:v-green}.
Define $\Gamma_\alpha:=[-\eps^{\alpha-1},\eps^{\alpha-1}]$.
Substituting \eqref{eq:u-app} into the second term on the right-hand side of \eqref{eq:v-green}, and using \eqref{eq:G-decomp}, we find that, at
\(x=\gamma(s)=\gamma(s_1+\eps\xi)\),
\begin{align*}
    &-\frac{\eps^3}{D}\int_\Gamma G\bigl(\gamma(s),\gamma(s')\bigr)u_T(s',T)\,ds'\\
    &=\frac{\eps^2\dot s_1}{D}\int_{\Gamma_1} G_\Gamma(s,s')U_0'\bigg(\frac{s'-s_1}{\eps}\bigg)\,ds'
    -\frac{\eps^2\dot s_2}{D}\int_{\Gamma_2} G_\Gamma(s,s')U_0'\bigg(\frac{s_2-s'}{\eps}\bigg)\,ds'\\
    &=\frac{\eps^3\dot s_1}{D}\int_{\Gamma_\alpha} G_\Gamma(s,s_1+\eps\xi')U_0'(\xi')\,d\xi'
    -\frac{\eps^3\dot s_2}{D}\int_{\Gamma_\alpha} G_\Gamma(s,s_2+\eps\xi')U_0'(-\xi')\,d\xi'\\
    &=-\frac{\eps^3\dot s_1}{\pi D}\int_{\Gamma_\alpha} \log|\gamma(s)-\gamma(s_1+\eps\xi')|U_0'(\xi')\,d\xi'
    +O(\eps^3)\\
    &=O(\eps^3).
\end{align*}
In the last equality, we used the exponential decay of \(U'\) and the integrability of the logarithmic singularity near the origin.
Similarly, the third term on the right-hand side of \eqref{eq:v-green} can be computed at
\(x=\gamma(s)=\gamma(s_1+\eps\xi)\) as follows:
\begin{align*}
    &\frac{\eps^2}{D}\int_\Gamma G\bigl(\gamma(s),\gamma(s')\bigr)u_{ss}(s',T)\,ds'\\
    &=\frac{1}{D}\int_{\Gamma_1} G_\Gamma(s,s')U_0''\bigg(\frac{s'-s_1}{\eps}\bigg)\,ds'
    +\frac{1}{D}\int_{\Gamma_2} G_\Gamma(s,s')U_0''\bigg(\frac{s_2-s'}{\eps}\bigg)\,ds'\\
    &=\frac{\eps}{D}\int_{\Gamma_\alpha}G_\Gamma(s,s_1+\eps\xi')U_0''(\xi')\,d\xi'
    +\frac{\eps}{D}\int_{\Gamma_\alpha} G_\Gamma(s,s_2+\eps\xi')U_0''(-\xi')\,d\xi'\\
    &=\frac{-\eps}{\pi D}\int_{\Gamma_\alpha}
    \log|\gamma(s)-\gamma(s_1+\eps\xi')|U_0''(\xi')\,d\xi'\\
      &\qquad
      +\frac{\eps}{D}\int_{\Gamma_\alpha} H(s,s_1+\eps\xi')U_0''(\xi')\,d\xi'
      +\frac{\eps}{D}\int_{\Gamma_\alpha} G_\Gamma(s,s_2+\eps\xi')U_0''(-\xi')\,d\xi'\\
    &=\frac{-\eps}{\pi D}\int_{\Gamma_\alpha}
    \log|\gamma(s_1+\eps\xi)-\gamma(s_1+\eps\xi')|U_0''(\xi')\,d\xi'\\
      &\qquad
      +\frac{\eps^2}{D}\partial_{s'}H(s,s_1)\int_{\Gamma_\alpha} \xi'\, U_0''(\xi')\,d\xi'
      +\frac{\eps^2}{D}\partial_{s'}G_\Gamma(s,s_2)\int_{\Gamma_\alpha} \xi'\, U_0''(-\xi')\,d\xi'
      +O(\eps^3)\\
    &=\frac{-\eps}{\pi D}\int_{\Gamma_\alpha}
    \log|\eps(\xi-\xi')|U_0''(\xi')\,d\xi'
    +O(\eps^{2\alpha+1})\\
      &\qquad
      +\frac{\eps^2}{D}\partial_{s'}H(s_1+\eps\xi,s_1)\int_\bR \xi'\, U_0''(\xi')\,d\xi'
      +\frac{\eps^2}{D}\partial_{s'}G_\Gamma(s_1+\eps\xi,s_2)\int_\bR \xi'\, U_0''(-\xi')\,d\xi'
      +O(\eps^3)\\
    &=\frac{-\eps}{\pi D}\int_\bR\log|\xi-\xi'|U_0''(\xi')\,d\xi'
      -\frac{\eps^2\Delta h(v_*)}{D}\partial_{s'}H(s_1,s_1)
      +\frac{\eps^2\Delta h(v_*)}{D}\partial_{s'}G_\Gamma(s_1,s_2)
      +o(\eps^2).
\end{align*}
In the fifth equality, we used the expansion
\[
\log|\gamma(s_1+\eps\xi)-\gamma(s_1+\eps\xi')|
=
\log|\eps(\xi-\xi')|
+
O(\eps^2|\xi-\xi'|^2).
\]
Therefore,
\begin{align*}
    v(\gamma(s_1+\eps\xi))-\ovl v(T)
    &=
    -\frac{\eps}{\pi D}\int_\bR \log|\xi-\xi'|U_0''(\xi')\,d\xi'\\
    &\quad
    +\frac{\eps^2\Delta h}{D}
    \bigl(\partial_{s'}G_\Gamma(s_1,s_2)-\partial_{s'}H(s_1,s_1)\bigr)
    +o(\eps^2).
\end{align*}
Similarly,
\begin{align*}
    v(\gamma(s_2-\eps\xi))-\ovl v(T)
    &=
    -\frac{\eps}{\pi D}\int_\bR \log|\xi-\xi'|U_0''(\xi')\,d\xi'\\
    &\quad
    +\frac{\eps^2\Delta h}{D}
    \bigl(-\partial_{s'}G_\Gamma(s_2,s_1)+\partial_{s'}H(s_2,s_2)\bigr)
    +o(\eps^2).
\end{align*}
Hence, by expanding $\ovl v(T)=\ovl v_0(T)+\eps\ovl v_1(T)+\eps^2\ovl v_2(T)$,
\begin{align}\label{eq:V-s1}
\begin{aligned}
    V_0^{(1)}(\gamma(s_1+\eps\xi),T)&=\ovl v_0(T),\\
    V_1^{(1)}(\gamma(s_1+\eps\xi),T)&=\ovl v_1(T)
    -\frac{1}{\pi D}\int_\bR \log|\xi-\xi'|\partial_{\xi'}^2U_0^{(1)}(\xi',T)\,d\xi',\\
    V_2^{(1)}(\gamma(s_1+\eps\xi),T)&=\ovl v_2(T)
    +\frac{\Delta h}{D}
    \bigl(\partial_{s'}G_\Gamma(s_1,s_2)-\partial_{s'}H(s_1,s_1)\bigr).
\end{aligned}
\end{align}
Near \(s=s_2\), we similarly obtain
\begin{align}\label{eq:V-s2}
\begin{aligned}
    V_0^{(2)}(\gamma(s_2-\eps\xi),T)&=\ovl v_0(T),\\
    V_1^{(2)}(\gamma(s_2-\eps\xi),T)&=\ovl v_1(T)
    -\frac{1}{\pi D}\int_\bR \log|\xi-\xi'|\partial_{\xi'}^2U_0^{(2)}(\xi',T)\,d\xi',\\
    V_2^{(2)}(\gamma(s_2-\eps\xi),T)&=\ovl v_2(T)
    +\frac{\Delta h}{D}
    \bigl(-\partial_{s'}G_\Gamma(s_2,s_1)+\partial_{s'}H(s_2,s_2)\bigr).
\end{aligned}
\end{align}

\begin{rem}
If one carries out the calculation including the \(O(\eps)\) term \(U_1\) in the expansion of \(u\), then the \(O(\eps^2)\) term of \(v-\ovl v\) contains a contribution from \(U_1\). However, this term is independent of \(s_1\) and \(s_2\), and hence does not affect the determination of the pulse location. We therefore omit it.
\end{rem}

To match with the outer solution, we impose on the first equation of \eqref{eq:inner-expansion} the boundary conditions
\[
U_0(-\infty,T)=h^-(v_0(T)),
\qquad
U_0(\infty,T)=h^+(v_0(T)).
\]
Multiplying this equation by \(\partial_\xi U_0\) and integrating over \(\R\), we obtain
\[
J(v_0(T))
=
\int_{h^-(v_0(T))}^{h^+(v_0(T))}f(u,v_0(T))\,du
=0.
\]
By Assumption~\ref{ass:f}--(2), \(v=v_*\) is an isolated root of \(J(v)=0\). Thus, we choose
\[
v_0(T)\equiv v_*.
\]
Together with \eqref{eq:ovlv0=v0}, this gives
\[
\ovl v_0(T)\equiv v_*.
\]
Let \(U\) be the solution of
\begin{align}\label{eq:def-U}
    U''+f(U,v_*)=0,
    \qquad
    U(-\infty)=h^-(v_*),\quad U(\infty)=h^+(v_*).
\end{align}
Then
\[
U_0^{(1)}(\xi)=U(\xi).
\]
Similarly, near \(s=s_2\), we have
\[
U_0^{(2)}(\xi)=U(\xi).
\]
Differentiating \eqref{eq:def-U}, we see that \(\mathcal{L}\) has the eigenfunction
\(\partial_\xi U\). Therefore, the solvability condition for the second equation in
\eqref{eq:inner-expansion} is
\begin{align*}
    0
    &=
    \langle-f_v(U,\ovl v_0)V_1,\partial_\xi U\rangle\\
    &=
    -\ovl v_1(T)J'(v_*)
    +\frac{1}{\pi D}
    \int_\bR\int_\bR
    f_v(U_0(\eta),v_*)\log|\eta-\eta' |
    \partial_{\eta'}^2 U_0(\eta')\,\partial_\eta U_0(\eta)
    \,d\eta'\,d\eta.
\end{align*}
For details, see \cite[Lemma~3.1]{hale2005}.
Hence,
\begin{align*}
     \ovl v_1(T)
     =
     \frac{1}{\pi DJ'(v_*)}
     \int_\bR\int_\bR
     f_v(U_0(\eta),v_*)\log|\eta-\eta' |
     \partial_{\eta'}^2 U_0(\eta')\,\partial_\eta U_0(\eta)
     \,d\eta'\,d\eta .
\end{align*}
Substituting \eqref{eq:u-app} and \(v=\ovl v+\tilde v\) into the conservation law
\eqref{eq:mass-conservation}, we obtain
\[
F(v_*,w_0(T))=O(\eps),
\]
where we have expanded
\[
w(T)=w_0(T)+O(\eps).
\]
Here, \(F\) is defined in \eqref{eq:F=0}. Therefore, $w_0(T)\equiv w_*$ and 
\begin{align}\label{eq:w=w*+O(eps)}
    w(T)=w_*+O(\eps).
\end{align}

Multiplying the third equations in \eqref{eq:inner-expansion} and \eqref{eq:inner-expansion2}
by \(\partial_\xi U_0\) and integrating, and using the fact that \(V_2^{(1)}\) and
\(V_2^{(2)}\) are spatially constant, we obtain
\begin{align*}
    -\dot s_1\kappa_*
    &=
    V_2^{(1)}
    \langle f_v(U_0^{(1)},v_*),\partial_\xi U_0^{(1)}\rangle
    +\langle f_*^{(1)},\partial_\xi U_0^{(1)}\rangle\\
    &=
    V_2^{(1)}J'(v_*)+Q_*,
\end{align*}
and
\begin{align*}
    \dot s_2\kappa_*
    &=
    V_2^{(2)}
    \langle f_v(U_0^{(2)},v_*),\partial_\xi U_0^{(2)}\rangle
    +\langle f_*^{(2)},\partial_\xi U_0^{(2)}\rangle\\
    &=
    V_2^{(2)}J'(v_*)+Q_*.
\end{align*}
Here, we have set
\[
\kappa_*:=\int_\bR (\partial_\xi U(\xi))^2\,d\xi,
\qquad
Q_*:=\langle f_*^{(1)},\partial_\xi U_0^{(1)}\rangle
=
\langle f_*^{(2)},\partial_\xi U_0^{(2)}\rangle.
\]
Here, by the choice of the oriented inner variables, we have
\(U_0^{(1)}=U_0^{(2)}=U\), and the equations determining \(U_1^{(i)}\) and \(V_1^{(i)}\) are identical for \(i=1,2\) by \eqref{eq:inner-expansion} and \eqref{eq:inner-expansion2}. Hence \(f_*^{(1)}=f_*^{(2)}\), and the quantity \(Q_*\) is independent of \(i\).
We also used the self-adjointness of \(\mathcal{L}\), which implies
$\langle\mathcal{L}U_2,\partial_\xi U_0\rangle=0$.
Subtracting the two equations and dividing by \(2\), we obtain from \eqref{eq:V-s1} and \eqref{eq:V-s2}
\begin{align}\label{eq:s0'-old}
\begin{aligned}
    \dot s_0
    &=
    \frac{\Delta h(v_*)J'(v_*)}{2\kappa_*D}
    \bigl(
    -\partial_{s'}G_\Gamma(s_2,s_1)
    +\partial_{s'}H(s_2,s_2)
    -\partial_{s'}G_\Gamma(s_1,s_2)
    +\partial_{s'}H(s_1,s_1)
    \bigr).
\end{aligned}
\end{align}
Define
\begin{align}\label{eq:E-def}
    E(s;w_*):=
    2G_\Gamma(s-w_*,s+w_*)
    -H(s-w_*,s-w_*)
    -H(s+w_*,s+w_*).
\end{align}
Then, by \eqref{eq:G-property} and \eqref{eq:w=w*+O(eps)}, \eqref{eq:s0'-old} can be rewritten at leading order as
\begin{align*}
    \dot s_0(T)
    =
    -\frac{\Delta h(v_*)J'(v_*)}{4\kappa_*D}
    E'(s_0(T);w_*).
\end{align*}
Returning to the original time variable \(t\), we obtain \eqref{eq:ODE2}.

\begin{rem}
The potential \(E\) can be rewritten as
\[
E(s_0;w_*)=
2G_\Gamma(s_1,s_2)-\sum_{i=1,2}H(s_i,s_i).
\]
Interestingly, this expression suggests that \(E\) can be interpreted as an interaction potential between the two transition layers. Indeed, the first term on the right-hand side has the form of a two-body interaction potential, whereas the second term has the form of a self-interaction potential.
\end{rem}

\begin{rem}\label{rem:symmetry-E}
Suppose that the domain is symmetric with respect to a reflection axis, and that \(\gamma(s_0)\) lies on this axis. Then the reflection maps \(\gamma(s)\) to \(\gamma(2s_0-s)\). Hence, near \(s_0\), we have $E(s)=E(2s_0-s)$.
Therefore, $E'(s_0)=0$,
and \(s_0\) is a critical point of \(E\). We call such critical points, which arise from symmetry, trivial critical points.
\end{rem}

\begin{rem}
For \(0<w_*\ll1\), we have
\[
E(s_0;w_*)
=
-\frac{2}{\pi}\log(2w_*)
+
\left\{
\frac{\kappa(s_0)^2}{3\pi}
-4\partial_s\partial_{s'}H(s_0,s_0)
\right\}w_*^2
+o(w_*^2),
\]
where \(\kappa(s)\) denotes the curvature of \(\Gamma\) at \(\gamma(s)\).
This should be contrasted with the boundary concentration phenomenon for
least-energy solutions of a semilinear Neumann problem in \cite{ni1993}, where the peak location is selected by critical points of the boundary curvature.
In the present problem, the potential \(E\) is not determined only by the local curvature. More precisely, the term $\partial_s\partial_{s'}H(s_0,s_0)$ reflects the global geometry of \(\Omega\).
\end{rem}

\section{Examples}\label{sec:examples}

In this section, we specify concrete domains \(\Omega\), compute the potential \(E\) derived in the previous section, and analyze the motion of pulse solutions governed by \eqref{eq:ODE2}. In the disk, the pulse location is arbitrary because of rotational symmetry. In contrast, when the symmetry of the domain is broken, the pulse location is selected by the domain geometry. As examples, we analyze two concrete domains in this section: a dumbbell-shaped domain and a perforated disk.

Throughout this section, as in Section~\ref{sec:ODE}, we assume that \(J'(v_*)>0\). Under this assumption, \eqref{eq:ODE2} is a gradient flow of the potential \(E\). Therefore, by studying the local minima and local maxima of \(E\), we can determine the possible locations of stationary pulse solutions and their stability.

\subsection{Conformal representation of \(E\)}\label{subsec:conformal}

We would like to compute the potential \(E\). However, computing $E$ directly for a general domain is generally difficult. Therefore, we compute \(E\) for domains generated by conformal mappings from the disk or an annulus. In what follows, we identify \(\bR^2\) with \(\bC\).

Such computations of Green’s functions based on conformal mappings have been used to analyze, for example, the dynamics of spike solutions in the Gierer–Meinhardt system \cite{kolokolnikov2003, kolokolnikov2004} and mean first passage time problems for narrow escape \cite{pillay2010}. Extensive analyses have been carried out for various domains, including dumbbell-shaped domains.
Although conformal mappings make it possible to compute Green’s functions in these problems, the equations determining the locations of spikes or other localized structures are often complicated. Consequently, analyses combining analytical and numerical approaches are frequently employed.
In the present study, however, remarkably, we are able to obtain the potential function \(E\) in a very simple and explicit form. Moreover, by considering conformal mappings from an annulus, we can derive an equally concise expression even for domains containing a hole.

We first consider the case of the disk.

\begin{prop}\label{prop:Conformal-rep-of-E}
Let
\[
D:=\{z\in\mathbb C:\ |z|<1\}.
\]
Let \(\Omega\subset\mathbb C\) be a simply connected domain, and let $f:D\to\Omega$ be a conformal map from \(D\) onto \(\Omega\). Assume further that
\(f\in C^2(\overline D)\) and
\[
f'(e^{i\theta})\neq0
\qquad
(\theta\in[0,2\pi)).
\]
Define
\[
\rho(\theta):=\bigl|f'(e^{i\theta})\bigr|,\qquad
S(\theta):=\int_0^\theta \rho(\tau)\,d\tau,\qquad
L:=S(2\pi),\qquad
\Theta:=S^{-1}.
\]
Fix \(0<w<L/2\), and set
\[
s_1=s_0-w,\qquad s_2=s_0+w,\qquad
\theta_i:=\Theta(s_i)\quad (i=1,2).
\]
Then
\[
E(s_0)
=
-\frac1\pi
\log\!\left(
4\,\rho(\theta_1)\rho(\theta_2)\sin^2\frac{\theta_2-\theta_1}{2}
\right).
\]
\end{prop}

\begin{rem}
The function \(\Theta\) in this proposition represents the argument of the point obtained by pulling back the boundary point \(\gamma(s)\in\Gamma\) to the unit circle by \(f\). Namely,
\[
f^{-1}(\gamma(s))=e^{i\Theta(s)}.
\]
\end{rem}

\begin{proof}
Set
\[
\widetilde G(z,\zeta):=G\bigl(f(z),f(\zeta)\bigr)
\qquad (z,\zeta\in D).
\]
Since \(f\) is conformal, the Laplacian transforms as
\[
\Delta_z(\cdot\circ f)=|f'(z)|^2(\Delta_x\cdot)\circ f.
\]
Thus, by the definition of \(G\), the function \(\widetilde G\) satisfies
\[
\Delta_z\widetilde G(z,\zeta)
=
\frac{|f'(z)|^2}{|\Omega|}
-|f'(z)|^2\delta(f(z)-f(\zeta))
=
\frac{|f'(z)|^2}{|\Omega|}
-\delta(z-\zeta)
\qquad (z\in D).
\]
Moreover, since \(f\) is conformal,
\[
\partial_n\widetilde G(\cdot,\zeta)=0
\qquad \text{on }\partial D.
\]

On the other hand, let \(G_D\) be the Green's function in the case \(\Omega=D\). Then \(G_D\) satisfies
\[
\Delta_z G_D(z,\zeta)=\frac1\pi-\delta(z-\zeta),
\qquad
\partial_n G_D(\cdot,\zeta)=0.
\]
Let \(Q\) be the solution of
\begin{align}\label{eq:Q-def}
\begin{cases}
\Delta Q(z)=\dfrac{|f'(z)|^2}{|\Omega|}-\dfrac1\pi & z\in D,\\[1mm]
\partial_n Q(z)=0 & z\in\partial D,\\[1mm]
\displaystyle\int_D Q(z)\,dz=0.
\end{cases}
\end{align}
Define
\[
F(z,\zeta):=\widetilde G(z,\zeta)-G_D(z,\zeta)-Q(z)-Q(\zeta).
\]
Then \eqref{eq:Q-def} implies
\[
\Delta_z F(z,\zeta)=0,\qquad
\partial_n F(\cdot,\zeta)=0.
\]

The solvability condition for \eqref{eq:Q-def} is satisfied, since
\[
\int_D\left(\frac{|f'(z)|^2}{|\Omega|}-\frac1\pi\right)\,dz
=
\frac1{|\Omega|}\int_D |f'(z)|^2\,dz-\frac{|D|}{\pi}
=
\frac{|\Omega|}{|\Omega|}-1
=
0.
\]
Hence \(Q\) exists. For fixed \(\zeta\), the function \(F(\cdot,\zeta)\) is harmonic in \(D\) and satisfies the homogeneous Neumann boundary condition. Therefore, it is constant. Thus, there exists a function \(c(\zeta)\) such that
\[
F(z,\zeta)=c(\zeta).
\]
However, both \(G\) and \(G_D\) are symmetric with respect to interchange of the variables, and \(Q(z)+Q(\zeta)\) is also symmetric. Hence \(F(z,\zeta)\) is symmetric. It follows that \(c(\zeta)\) is independent of \(\zeta\), and hence is a constant \(C_0\). Therefore,
\[
G\bigl(f(z),f(\zeta)\bigr)
=
G_D(z,\zeta)+Q(z)+Q(\zeta)+C_0.
\]

Next, using the explicit formula for \(G_D\) \cite[(4.1)]{kolokolnikov2005},
\[
G_D(z,\zeta)
=
-\frac1{2\pi}\log|z-\zeta|
-\frac1{2\pi}\log|1-z\overline{\zeta}|
+\frac{|z|^2+|\zeta|^2}{4\pi}
-\frac{3}{8\pi},
\]
we find that, when \(|z|=|\zeta|=1\),
\[
|1-z\overline{\zeta}|=|z-\zeta|.
\]
Thus,
\[
G_D(e^{i\theta},e^{i\phi})
=
-\frac1\pi\log|e^{i\theta}-e^{i\phi}|+\frac1{8\pi}.
\]
Absorbing the constant term into \(C_0\), we obtain
\[
G\bigl(f(e^{i\theta}),f(e^{i\phi})\bigr)
=
-\frac1\pi\log|e^{i\theta}-e^{i\phi}|
+Q(e^{i\theta})+Q(e^{i\phi})+C_0.
\]

Now, for \(s,t\in[0,L)\), set
\[
\theta=\Theta(s),\qquad \phi=\Theta(t).
\]
Then
\[
\gamma(s)=f(e^{i\theta}),\qquad
\gamma(t)=f(e^{i\phi}).
\]
By the definition of \(H\), we have
\begin{align*}
H(s,t)
&=
G\bigl(f(e^{i\theta}),f(e^{i\phi})\bigr)
+\frac1\pi\log|f(e^{i\theta})-f(e^{i\phi})|\\
&=
Q(e^{i\theta})+Q(e^{i\phi})+C_0
+\frac1\pi\log
\frac{|f(e^{i\theta})-f(e^{i\phi})|}
{|e^{i\theta}-e^{i\phi}|}.
\end{align*}
Letting \(t\to s\), we obtain
\[
\frac{|f(e^{i\theta})-f(e^{i\phi})|}
{|e^{i\theta}-e^{i\phi}|}
\longrightarrow
|f'(e^{i\theta})|
=
\rho(\theta).
\]
Therefore,
\[
H(s,s)=2Q(e^{i\theta})+C_0+\frac1\pi\log\rho(\theta).
\]
Hence, by \eqref{eq:E-def},
\begin{align*}
E(s_0)
&=
2G\bigl(\gamma(s_1),\gamma(s_2)\bigr)
-H(s_1,s_1)-H(s_2,s_2)\\
&=
-\frac2\pi\log|e^{i\theta_1}-e^{i\theta_2}|
-\frac1\pi\log\rho(\theta_1)
-\frac1\pi\log\rho(\theta_2).
\end{align*}
Since
\[
|e^{i\theta_1}-e^{i\theta_2}|
=
2\left|\sin\frac{\theta_2-\theta_1}{2}\right|,
\]
and since \(0<w<L/2\), we may take \(\theta_2-\theta_1\in(0,2\pi)\), so that the right-hand side is positive. Consequently,
\[
E(s_0)
=
-\frac1\pi
\log\!\left(
4\,\rho(\theta_1)\rho(\theta_2)\sin^2\frac{\theta_2-\theta_1}{2}
\right).
\]
\end{proof}

Next, we consider the case of an annulus.

\begin{prop}\label{prop:Conformal-rep-of-E-annulus}
Let
\[
A_a:=\{z\in\mathbb C:\ a<|z|<1\},
\qquad 0<a<1.
\]
Let \(\Omega\subset\mathbb C\) be a doubly connected domain, and let
$f:A_a\to\Omega$ be a conformal map from \(A_a\) onto \(\Omega\), where we identify
\(\bR^2\) with \(\bC\). Assume further that \(f\in C^2(\overline{A_a})\) and
\[
f'(z)\neq0
\qquad
(z\in \partial A_a).
\]
Moreover, define
\[
\rho(\theta):=\bigl|f'(e^{i\theta})\bigr|,
\qquad
S(\theta):=\int_0^\theta \rho(\tau)\,d\tau,
\qquad
L:=S(2\pi),
\qquad
\Theta:=S^{-1}.
\]
Fix \(0<w<L/2\), and set
\[
s_1=s_0-w,\qquad s_2=s_0+w,\qquad
\theta_i:=\Theta(s_i)\quad (i=1,2).
\]
Then
\[
E(s_0)
=
C_w
-\frac1\pi
\log\!\left(
4\,\rho(\theta_1)\rho(\theta_2)\sin^2\frac{\theta_2-\theta_1}{2}
\right)
+\frac{4}{\pi}\sum_{n=1}^\infty
\frac{a^{2n}}{n(1-a^{2n})}\cos n(\theta_2-\theta_1),
\]
where \(C_w\) is a constant independent of \(s_0\).
\end{prop}

\begin{proof}
Let \(Q\) be the solution of
\begin{align*}
\begin{cases}
\Delta Q(z)=\dfrac{|f'(z)|^2}{|\Omega|}-\dfrac1{\pi(1-a^2)}
& z\in A_a,\\[1mm]
\partial_n Q(z)=0
& z\in\partial A_a,\\[1mm]
\displaystyle\int_{A_a}Q(z)\,dz=0,
\end{cases}
\end{align*}
and let \(G_{A_a}\) be the Neumann Green's function on the annulus \(A_a\).
As in the disk case, one can show that
\[
G\bigl(f(z),f(\zeta)\bigr)
=
G_{A_a}(z,\zeta)+Q(z)+Q(\zeta)+C_0.
\]
The boundary trace of the Neumann Green's function on the annulus
$G_{A_a}(e^{i\theta},e^{i\phi})$
is given by
\[
G_{A_a}(e^{i\theta},e^{i\phi})
=
C_0-\frac1\pi\log|e^{i\theta}-e^{i\phi}|
+\frac{2}{\pi}\sum_{n=1}^\infty
\frac{a^{2n}}{n(1-a^{2n})}\cos n(\theta-\phi)
\]
\cite[Appendix B]{mangeat2019}.
Substituting this formula into the above decomposition, we obtain, exactly as in the disk case,
\begin{align*}
H(s,t)
&=
Q(e^{i\theta})+Q(e^{i\phi})+C_0\\
&\quad
+\frac1\pi
\log\frac{|f(e^{i\theta})-f(e^{i\phi})|}
{|e^{i\theta}-e^{i\phi}|}
+\frac{2}{\pi}\sum_{n=1}^\infty
\frac{a^{2n}}{n(1-a^{2n})}\cos n(\theta-\phi).
\end{align*}
Repeating the same calculation as in the disk case, we obtain
\[
E(s_0)
=
C_w
-\frac1\pi
\log\!\left(
4\,\rho(\theta_1)\rho(\theta_2)\sin^2\frac{\theta_2-\theta_1}{2}
\right)
+\frac{4}{\pi}\sum_{n=1}^\infty
\frac{a^{2n}}{n(1-a^{2n})}\cos n(\theta_2-\theta_1),
\]
for some constant \(C_w\) independent of \(s_0\).
\end{proof}

\subsection{Case of a dumbbell-shaped domain}\label{subsec:dumbbell}
We next use Proposition~\ref{prop:Conformal-rep-of-E} to compute \(E\) for a dumbbell-shaped domain and investigate its bifurcation structure and dynamics. The map
\[
f(z)=\frac{(1-k)z}{1-kz^2}\qquad (z\in D,\ 0\le k<1)
\]
is conformal, and the shape of \(f(D)\) changes with the parameter \(k\), as shown in Figure~\ref{fig:dumbbell}. When \(k=0\), the domain is the unit disk, while as \(k\to1\), the domain approaches a pair of tangent disks of radius \(1/2\). Moreover, the domain is convex for
$k<k_c:=(\sqrt{2}-1)^2$
and non-convex for \(k>k_c\).

\begin{prop}\label{prop:E-for-f-rational}
Let \(0\le k<1\), and set
\[
f(z)=\frac{(1-k)z}{1-kz^2}\qquad (z\in D).
\]
Furthermore, set $\theta_i:=\Theta(s_i)$ for $i=1,2$.
Then
\begin{align}\label{eq:dumbbell-E}
E(s_0)
=
-\frac1\pi
\log\!\left[
4(1-k)^2\sin^2\frac{\theta_2-\theta_1}{2}\,
\frac{
\sqrt{\bigl((1-k)^2+4k\cos^2\theta_1\bigr)
      \bigl((1-k)^2+4k\cos^2\theta_2\bigr)}
}{
\bigl((1-k)^2+4k\sin^2\theta_1\bigr)
\bigl((1-k)^2+4k\sin^2\theta_2\bigr)
}
\right].
\end{align}
\end{prop}

\begin{proof}
First, we compute
\[
f'(z)
=
(1-k)\frac{(1-kz^2)-z(-2kz)}{(1-kz^2)^2}
=
\frac{(1-k)(1+kz^2)}{(1-kz^2)^2}.
\]
Therefore,
\[
\rho(\theta)
=
|f'(e^{i\theta})|
=
(1-k)\frac{|1+ke^{2i\theta}|}{|1-ke^{2i\theta}|^2}.
\]
Since
\[
|1+ke^{2i\theta}|^2
=
1+k^2+2k\cos2\theta
=
(1-k)^2+4k\cos^2\theta,
\]
and
\[
|1-ke^{2i\theta}|^2
=
1+k^2-2k\cos2\theta
=
(1-k)^2+4k\sin^2\theta,
\]
we obtain
\[
\rho(\theta)
=
\frac{(1-k)\sqrt{(1-k)^2+4k\cos^2\theta}}
{(1-k)^2+4k\sin^2\theta}.
\]
Thus, \eqref{eq:dumbbell-E} follows from Proposition~\ref{prop:Conformal-rep-of-E}.
\end{proof}

%%%%%
\begin{figure}[t]
\centering
\includegraphics[width=0.45\textwidth]{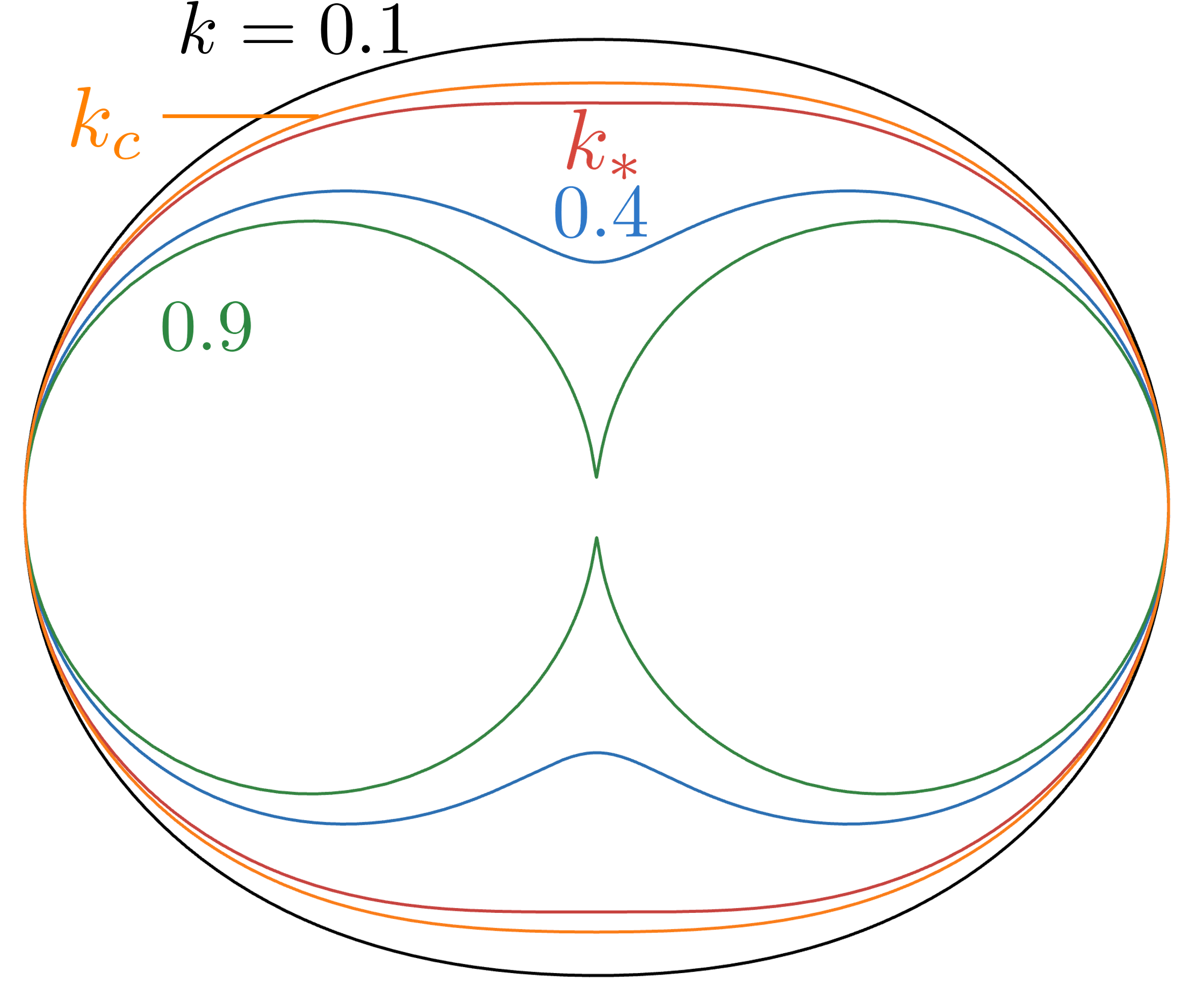}
\caption{
Dumbbell-shaped domains for different values of $k$.
}
\label{fig:dumbbell}
\end{figure}
%%%%%
\begin{thm}\label{thm:dumbbell}
 Since the perimeter \(L\) depends on \(k\), we write \(L=L(k)\). Fix \(0<w<L(k)/2\). By symmetry, it suffices to consider \(s\in[0,L(k)/4]\).
Set
\[
k_*:=\frac{5+2\sqrt{7}-2\sqrt{11+5\sqrt{7}}}{3}\sim0.148.
\]
Moreover, define
\[
g(\mu):=\frac{2+2\mu-3\mu^2}{3\mu+2},
\qquad
w_b(k):=\int_0^{\arcsin{\sqrt{g((1-k)^2/4k)}}}\rho(\theta)\,d\theta\quad \text{for } k\in(k_*,1).
\]
Then the following statements hold.
\begin{enumerate}
\item
For \(0<k<1\),
\[
s=0,\qquad s=L(k)/4
\]
are critical points of \(E\), and \(s=L(k)/4\) is a strict local minimum point.

\item
If \(0<k\le k_*\), then the only critical points of \(E\) in
\([0,L(k)/4]\) are
\[
s=0,\qquad s=L(k)/4.
\]
Moreover, \(s=0\) is a strict local maximum point.

\item
If \(k_*<k<1\) and
\[
w_b(k)<w<L(k)/2-w_b(k),
\]
then the only critical points of \(E\) in \([0,L(k)/4]\) are
\[
s=0,\qquad s=L(k)/4.
\]
Moreover, \(s=0\) is a strict local maximum point.

\item
If \(k_*<k<1\) and
\[
0<w<w_b(k)
\qquad\text{or}\qquad
L(k)/2-w_b(k)<w<L(k)/2,
\]
then \(s=0\) is a strict local minimum point. In particular, there exists at least one strict local maximum point in \((0,L(k)/4)\).

\item
The function \(w_b(k)\) is strictly increasing for \(k\in(k_*,1)\), and
\[
w_b(k)\to0\quad (k\to k_*),
\qquad
\frac{w_b(k)}{L(k)}\to\frac14\quad (k\to1).
\]
\end{enumerate}
\end{thm}

\begin{proof}
See Appendix~\ref{app:proof-of-dumbbell}.
\end{proof}

We first note the role of the deformation parameter \(k\). In the circular case
\(k=0\), the domain is rotationally symmetric, and hence no preferred location
is selected by symmetry; equivalently, \(E(\cdot;w_*)\) is constant. When
\(k>0\) is small, the dumbbell-shaped domain may be regarded as a small
perturbation of the disk, and Theorem~\ref{thm:dumbbell}--(2) shows that no
nontrivial equilibrium appears in the interval \((0,L(k)/4)\).
In this regime, the deformation only selects the symmetric locations
\(s=0\) and \(s=L(k)/4\). However, once \(k\) becomes sufficiently large,
nontrivial equilibria appear between these two symmetric locations.

In particular, fix \(0<w<\pi/2\). Together with \eqref{eq:ODE2},
Theorem~\ref{thm:dumbbell} implies that a subcritical pitchfork
bifurcation occurs at
\[
k=(w_b)^{-1}(w)=:k_b(w)\in(k_*,1).
\]
The left panel of Figure~\ref{fig:sim-dumbbell} schematically illustrates this bifurcation. As predicted by the analysis, when \(k>k_b(w)\), a nontrivial unstable equilibrium appears, and the equilibria \(s=0\) and \(s=L(k)/4\) become bistable. The right panels show numerical simulation results for \eqref{eq:model} with different initial conditions. In the upper and lower panels, the solution remains near the points corresponding to \(s=0\) and \(s=L(k)/4\), respectively. Plots of \(E\) are given in Appendix~\ref{app:plots-of-potential}.
  The comparison between this theoretical results and direct numerical simulations are available in Appendix~\ref{app:comparison}.

\subsection{Case of a perforated disk}\label{subsec:hole}

Next, we compute \(E\) for a disk with one circular hole by using
Proposition~\ref{prop:Conformal-rep-of-E-annulus}. We call this domain
a perforated disk. In this subsection, \(\Gamma\) denotes only the outer
boundary, on which the surface variable \(u\) is defined. On the boundary
of the hole, \(u\) is not defined and we impose only the homogeneous
Neumann condition for the bulk variable \(v\). Since a perforated disk is
not simply connected, we use a conformal map from an annulus.

\begin{rem}
Although the formal derivation of the interface ODEs in
Section~\ref{sec:ODE} was written for the case where the boundary consists of
a single closed curve, the same derivation applies to the present perforated
disk. 
\end{rem}

\begin{prop}\label{prop:E-perforated-disk}
Let
\[
\Omega
=
\{z\in\mathbb C:\ |z|<1\}\setminus \overline{B(c,r)}
\qquad
(0<r<1,\ 0< c<1-r),
\]
where \(B(c,r)\) is the disk centered at \((c,0)\) with radius \(r\).
Let \(a,b\in(0,1)\) be the unique pair determined by
\[
c=\frac{b(1-a^2)}{1-a^2b^2},
\qquad
r=\frac{a(1-b^2)}{1-a^2b^2}.
\]
Then
\begin{align}\label{eq:hole-E}
E(s_0)
&=
C_w
-\frac{2}{\pi}\sum_{m=1}^\infty
\log\!\left(
1+
\frac{4a^{2m}(1-b^2)^2\sin^2 w}
{(1-a^{2m})^2
\bigl(1-2b\cos(s_0-w)+b^2\bigr)
\bigl(1-2b\cos(s_0+w)+b^2\bigr)}
\right).
\end{align}
\end{prop}

%%%%%
\begin{figure}[t]
\centering
\includegraphics[width=0.95\textwidth]{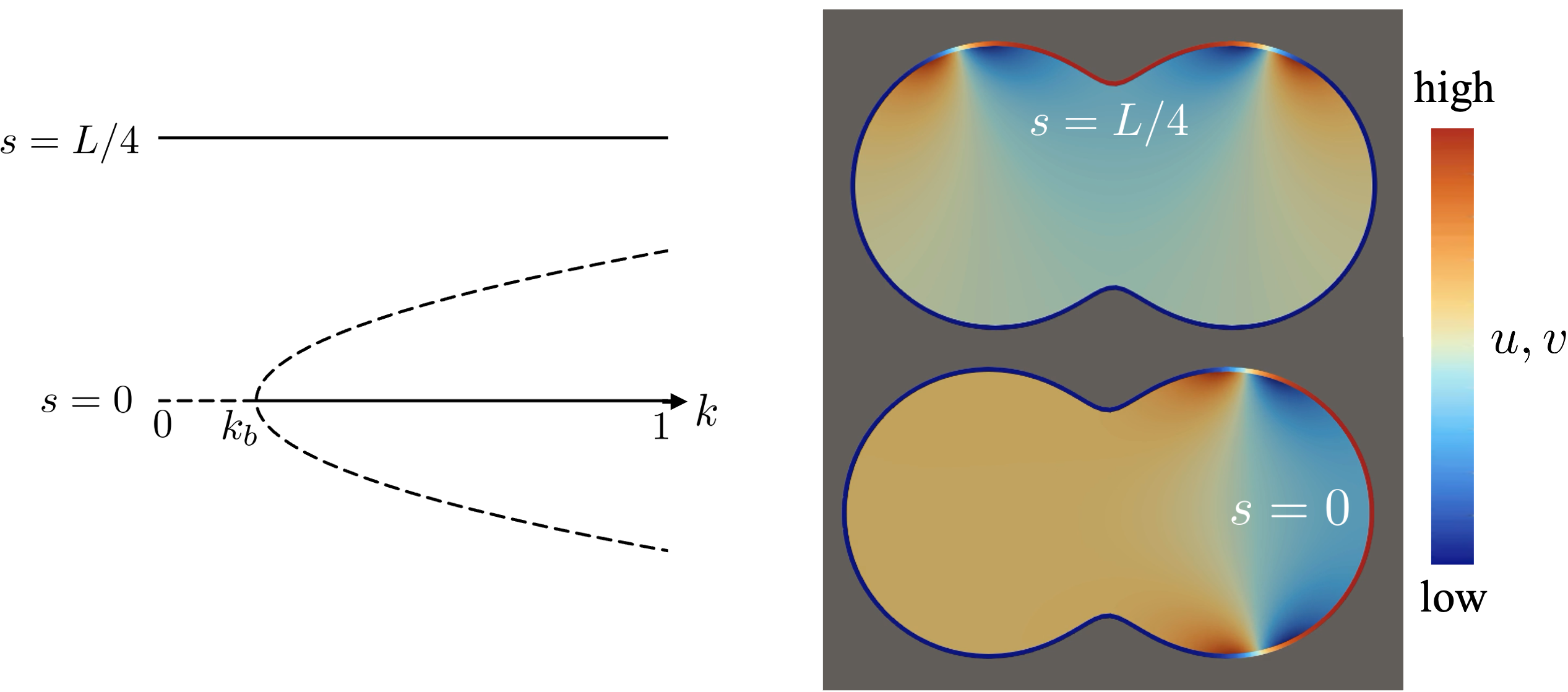}
\caption{
Left: schematic diagram of the pitchfork bifurcation in Theorem~\ref{thm:dumbbell}. Solid and dotted curves represent stable and unstable equilibria, respectively. Right: snapshots of numerical simulation results after sufficiently long time. The same value \(k=0.44\) is used in the upper and lower simulations, while only the initial condition is changed. See Appendix~\ref{app:numerical} for the details of the numerical simulations.
}
\label{fig:sim-dumbbell}
\end{figure}
%%%%%

\begin{proof}
In this case,
\[
f(z)=\frac{z+b}{1+bz}
\]
is a conformal map from
\[
A_a=\{z\in\mathbb C:\ a<|z|<1\}
\]
onto \(\Omega\). Moreover,
\[
f'(z)=\frac{1-b^2}{(1+bz)^2},
\qquad
f^{-1}(z)=\frac{z-b}{1-bz}.
\]
Hence,
\[
e^{i\Theta(s)}
=
f^{-1}(e^{is})
=
\frac{e^{is}-b}{1-be^{is}}.
\]
Also,
\[
\rho(\theta)=|f'(e^{i\theta})|
=
\frac{1-b^2}{|1+be^{i\theta}|^2}
=
\frac{1-b^2}{1+2b\cos\theta+b^2}.
\]
Therefore,
\[
\rho(\Theta(s))
=
\frac{|1-be^{is}|^2}{1-b^2}
=
\frac{1-2b\cos s+b^2}{1-b^2}.
\]
Furthermore,
\begin{align}\label{eq:sin^2Theta}
\sin^2\frac{\Theta(s_2)-\Theta(s_1)}{2}
=
\frac{(1-b^2)^2\sin^2 w}
{\bigl(1-2b\cos(s_0-w)+b^2\bigr)
 \bigl(1-2b\cos(s_0+w)+b^2\bigr)}
=
\frac{\sin^2w}{\rho(\Theta(s_1))\rho(\Theta(s_2))}.
\end{align}
Thus, by Proposition~\ref{prop:Conformal-rep-of-E-annulus}, we obtain
\begin{align*}
E(s_0)
&=
C_w
+\frac{4}{\pi}\sum_{n=1}^\infty
\frac{a^{2n}}{n(1-a^{2n})}
\cos\!\Bigl(n\bigl(\Theta(s_0+w)-\Theta(s_0-w)\bigr)\Bigr).
\end{align*}
Since \(0<a<1\), for \(0\le\phi<2\pi\) we have
\begin{align*}
\sum_{n=1}^\infty \frac{a^{2n}}{n(1-a^{2n})}\cos(n\phi)
=
\sum_{m=1}^\infty\sum_{n=1}^\infty \frac{a^{2mn}}{n}\cos(n\phi)=
-\frac12\sum_{m=1}^\infty
\log\!\bigl(1-2a^{2m}\cos\phi+a^{4m}\bigr).
\end{align*}
Therefore,
\begin{align*}
E(s)
&=
C_w
-\frac{2}{\pi}\sum_{m=1}^\infty
\log\!\Bigl(
1-2a^{2m}\cos\bigl(\Theta(s+w)-\Theta(s-w)\bigr)+a^{4m}
\Bigr).
\end{align*}
Using the identity
\[
1-2r\cos\phi+r^2=(1-r)^2+4r\sin^2\frac{\phi}{2},
\]
we obtain
\begin{align*}
&1-2a^{2m}\cos\bigl(\Theta(s+w)-\Theta(s-w)\bigr)+a^{4m}\\
&\qquad
=
(1-a^{2m})^2
\left(
1+
\frac{4a^{2m}}{(1-a^{2m})^2}
\sin^2\frac{\Theta(s+w)-\Theta(s-w)}{2}
\right).
\end{align*}
Substituting \eqref{eq:sin^2Theta} into this expression gives \eqref{eq:hole-E}.
\end{proof}

\begin{thm}\label{thm:hole}
For $0<c<1-r$, the critical points of \(E\) are classified as follows.
\begin{enumerate}
\item
If
\[
0<w<\arccos\frac{2b}{1+b^2},
\]
then the only critical points are \(s=0,\pi\). Moreover, \(s=0\) is a strict local minimum point, and \(s=\pi\) is a strict local maximum point.

\item
If
\[
\arccos\frac{2b}{1+b^2}
<
w
<
\pi-\arccos\frac{2b}{1+b^2},
\]
then the critical points are
\[
s=0,\quad \pi,\quad
\pm\arccos\!\left(\frac{1+b^2}{2b}\cos w\right).
\]
Moreover,
\[
s=\pm\arccos\!\left(\frac{1+b^2}{2b}\cos w\right)
\]
are strict local minimum points, while \(s=0,\pi\) are strict local maximum points.

\item
If
\[
\pi-\arccos\frac{2b}{1+b^2}<w<\pi,
\]
then the only critical points are \(s=0,\pi\). Moreover, \(s=\pi\) is a strict local minimum point, and \(s=0\) is a strict local maximum point.

\item
For any fixed \(0<r<1\), the quantity
\[
\arccos\frac{2b}{1+b^2}
\]
is strictly decreasing with respect to \(c\), and satisfies
\[
\arccos\frac{2b}{1+b^2}\to\frac\pi2\quad (c\to0),
\qquad
\arccos\frac{2b}{1+b^2}\to0\quad (c\to1-r).
\]
\end{enumerate}
\end{thm}

\begin{proof}
See Appendix~\ref{app:proof-of-hole}.
\end{proof}

%%%%%
\begin{figure}[t]
\centering
\includegraphics[width=0.85\textwidth]{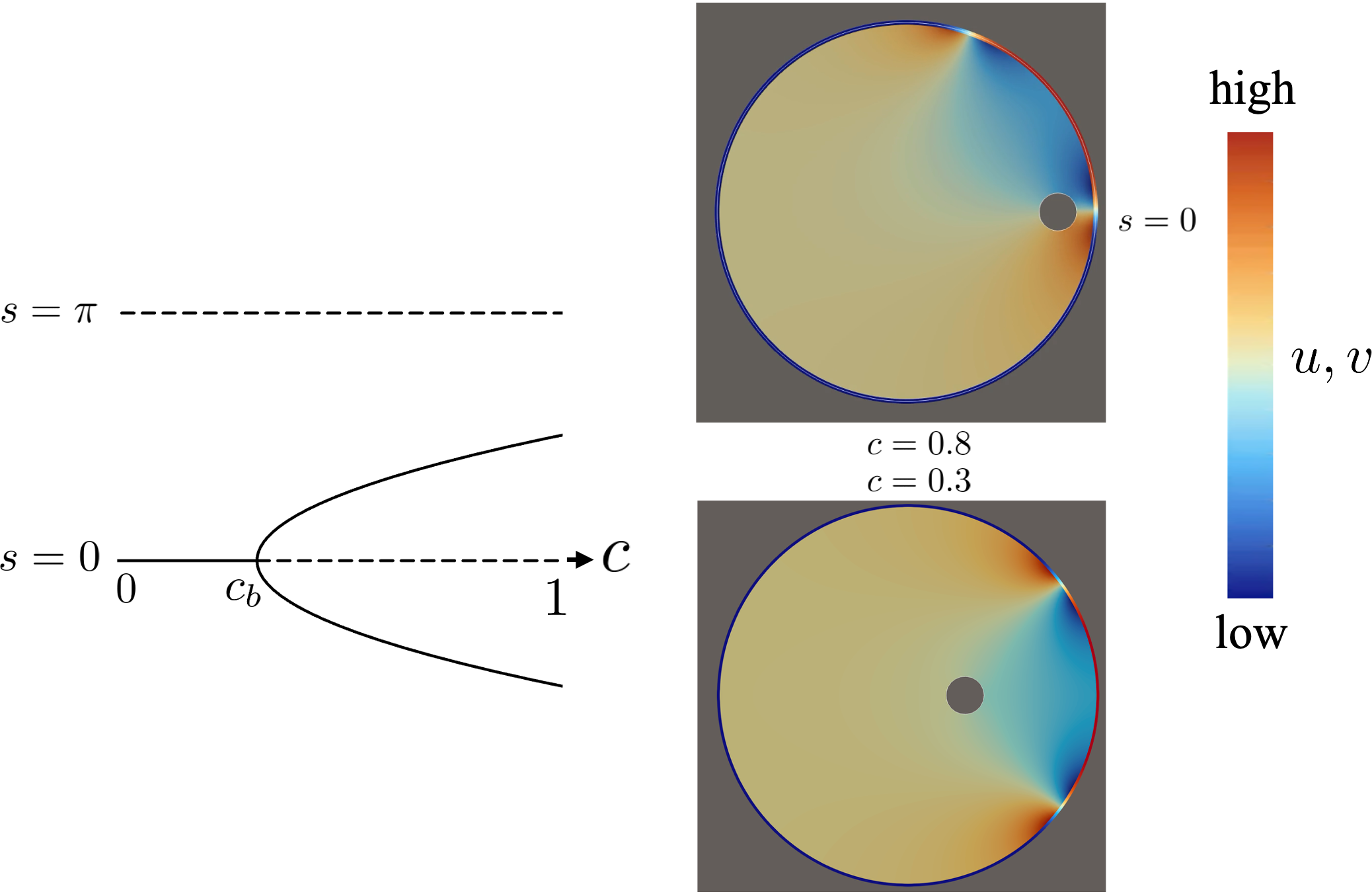}
\caption{
Left: schematic diagram of the pitchfork bifurcation in Theorem~\ref{thm:hole}. Solid and dotted curves represent stable and unstable equilibria, respectively. Right: snapshots of numerical simulation results after sufficiently long time. The same initial condition and the same parameter \(r=0.1\) are used in the upper and lower simulations, while only the value of \(c\) is changed. See Appendix~\ref{app:numerical} for the details of the numerical simulations.
}
\label{fig:sim-hole}
\end{figure}
%%%%%

We first note the role of the eccentricity of the hole. In the case \(c=0\), 
the domain is a concentric annulus, and hence no preferred location is selected 
by symmetry; equivalently, \(E(\cdot;w_*)\) is constant. When \(c>0\) is small, 
the eccentric annulus may be regarded as a small perturbation of the concentric 
one, and Theorem~\ref{thm:hole} shows that no nontrivial stable equilibrium 
appears for fixed \(0<w<\pi/2\). In this regime, the eccentricity only selects 
the symmetric equilibrium \(s=0\) as a stable location. However, once \(c\) 
becomes sufficiently large, a pair of nontrivial stable equilibria bifurcates 
from \(s=0\).

Together with \eqref{eq:ODE2}, this result shows that, for fixed
\(0<r<1\) and \(0<w<\pi/2\),
there exists \(0<c_b(r,w)<1-r\) such that a supercritical pitchfork bifurcation 
occurs at \(c=c_b(r,w)\). For \(c>c_b(r,w)\), nontrivial stable equilibria
\[
s=\pm\arccos\!\left(\frac{1+b^2}{2b}\cos w\right)
\]
exist. The left panel of Figure~\ref{fig:sim-hole} schematically illustrates this bifurcation. The right panels show numerical simulation results for \eqref{eq:model} on the perforated disk. As predicted by the analysis, the upper-right panel shows that the solution remains near the point corresponding to
\[
s=\arccos\!\left(\frac{1+b^2}{2b}\cos w\right),
\]
whereas the lower-right panel shows that the solution remains near \(s=0\). Plots of \(E\) are given in Appendix~\ref{app:plots-of-potential}. 
  The comparison between this theoretical results and direct numerical simulations are available in Appendix~\ref{app:comparison}.

\section{Discussion}\label{sec:conclusion}

In this paper, we studied the effect of domain geometry on pulse solutions of \eqref{eq:model}. By formally deriving equations of motion for the interfaces of pulse solutions of \eqref{eq:model} on two different time scales, we obtained a condition under which wave-pinning occurs stably. We further showed that the pulse location follows the gradient flow of a potential function determined by the domain geometry and the half-width of the pulse. In addition, by specifying concrete domain geometries and analyzing the resulting reduced equations, we classified the pulse dynamics and showed, in particular, that nontrivial dynamics can arise through supercritical or subcritical pitchfork bifurcations. The examples of the dumbbell-shaped domain and the perforated disk show that nontrivial pulse dynamics does not arise from small perturbations of the disk or the concentric annulus, but appears only when the symmetry breaking is sufficiently strong. We also supported and visualized the analytical results by numerical simulations.

Since \eqref{eq:BS} involves variables defined both in the bulk and on the boundary, it is not straightforward to apply conventional perturbation techniques. To overcome this difficulty, we combined two approaches. The first is a conventional singular perturbation approach: for the boundary variable, we decomposed the problem into inner and outer regions near the transition layers, and constructed the inner solution using stretched coordinates. The second concerns the treatment of the bulk variable. We used an integral representation in terms of the Green's function and the boundary variable, which allowed us to compute the bulk contribution through detailed asymptotic calculations.

Boundary conditions in which variables evolve in time on the boundary, as in \eqref{eq:BS}, are called dynamic boundary conditions (DBCs). They have recently attracted considerable attention because they naturally describe physical and biological phenomena involving state variables evolving on boundaries \cite{diez2024,duda2023,handy2021,scheel2021}. In this paper, we investigated the dynamics of a problem with a DBC by using singular perturbation methods. The analytical approach developed here is not restricted to \eqref{eq:BS}; it is expected to provide a basis for applying singular perturbation methods to other problems with DBCs, such as moving boundary problems with dynamic boundary conditions \cite{caetano2025,macdonald2016} and the Cahn--Hilliard equation with dynamic boundary conditions \cite{goldstein2011,kagawa2024}. Since many existing analytical methods cannot be directly applied to problems with DBCs, their mathematical analysis is often difficult. The framework proposed in this paper is therefore expected to be useful for future singular perturbation analysis of problems involving DBCs.

Finally, we conclude this section by discussing future research directions. In this paper, we formally derived the equations of motion for pulse solutions and discussed the possible locations and stability of stationary pulse solutions. A rigorous proof of the existence and stability of such stationary pulse solutions remains future work. For such a rigorous analysis, it will be necessary to refine the asymptotic calculations carried out in this paper and to construct approximate solutions more precisely. In particular, unlike standard singular perturbation problems, the logarithmic singularity of the Green's function causes algebraic decay of the inner solution, and therefore matching in the overlap region must be handled carefully. Moreover, proving the existence and stability of exact solutions requires a detailed spectral analysis of the linearized operator around the approximate solution. This is an important problem for future study.

In this study, we computed the Neumann Green's function using conformal mappings. However, for domains in which the Neumann Green's function has already been obtained in the other ways such as coordinate transformations or perturbation methods, the potential function \(E\) can be computed directly. For example, the Neumann Green's function for an elliptical domain was derived in \cite[(5.21a)]{Iyaniwura2021}, while that for a perturbed disk was obtained in \cite[Appendix A]{pillay2010}. It would be interesting to investigate the pulse dynamics in these geometries using these explicit representations. Moreover, a fast and accurate numerical method for computing Neumann Green's functions in general two-dimensional domains, based on a boundary integral formulation, has recently been proposed in \cite{chakraborty2025a}. Combining this numerical method with the results of the present study would make it possible to investigate pulse dynamics in general two-dimensional domains.

A possible extension of the present study is to investigate the case where \(\Omega\) is a three-dimensional domain. The three-dimensional problem is more challenging than the two-dimensional case considered here, since the singularity structure of the Green's function is different and the interface becomes a free boundary. In this connection, for the shadow system corresponding to \(D\to\infty\), the interface evolution has been formally shown to be governed by an area-preserving geodesic curvature flow on the surface \cite{miller2023}. When \(D\) is finite, identifying the geometric quantities that determine the interface dynamics would be a challenging but interesting problem. Recently, a numerical method for computing the Neumann Green's function in three dimensions has also been proposed, in which difficulties associated with its singularity structure and anisotropic effects arising from the principal curvatures are appropriately treated. This development may provide an important tool for the numerical investigation of the corresponding three-dimensional problem.

In addition, the biological interpretation of the results obtained in this paper is an interesting issue. The system \eqref{eq:BS} has been used as a bulk--surface model describing membrane--crytosol shuttling of Cdc42/Rho GTPase and cell polarity formation in the yeast \textit{Saccharomyces cerevisiae}. The results of this study show that domain geometry can affect the selection of localized patterns on the membrane and, in some cases, can generate bistability or nontrivial stable locations. Therefore, an important future direction is to compare these results with experiments and investigate how cell shape affects the position selection of Cdc42 clusters and polarity formation.

\section*{Statements and Declarations}

\noindent\textbf{Funding}
This work was supported by JST SPRING, Grant Number JPMJSP2119.

\noindent\textbf{Competing interests}
The author declares no competing interests.

\noindent\textbf{Ethics approval and consent to participate}
Not applicable.

\noindent\textbf{Consent for publication}
Not applicable.

\noindent\textbf{Data availability}
The data generated during the current study are available from the author upon reasonable request.

\noindent\textbf{Materials availability}
Not applicable.

\noindent\textbf{Use of generative AI}
OpenAI's ChatGPT was used for English-language editing and as a discussion tool during the preparation of this manuscript. In particular, discussions with ChatGPT helped the author identify the connection between the potential function \(E\) and the excursion Poisson kernel stated in Theorem~\ref{cor:Poisson}. All mathematical statements, derivations, references, and conclusions were independently verified by the author.

\appendix

\section{Plots of the potential \(E\)}\label{app:plots-of-potential}
Plots of the potential \(E\) are given in Figure~\ref{fig:E-plot}.
\begin{figure}[t]
\centering
\includegraphics[width=0.95\textwidth]{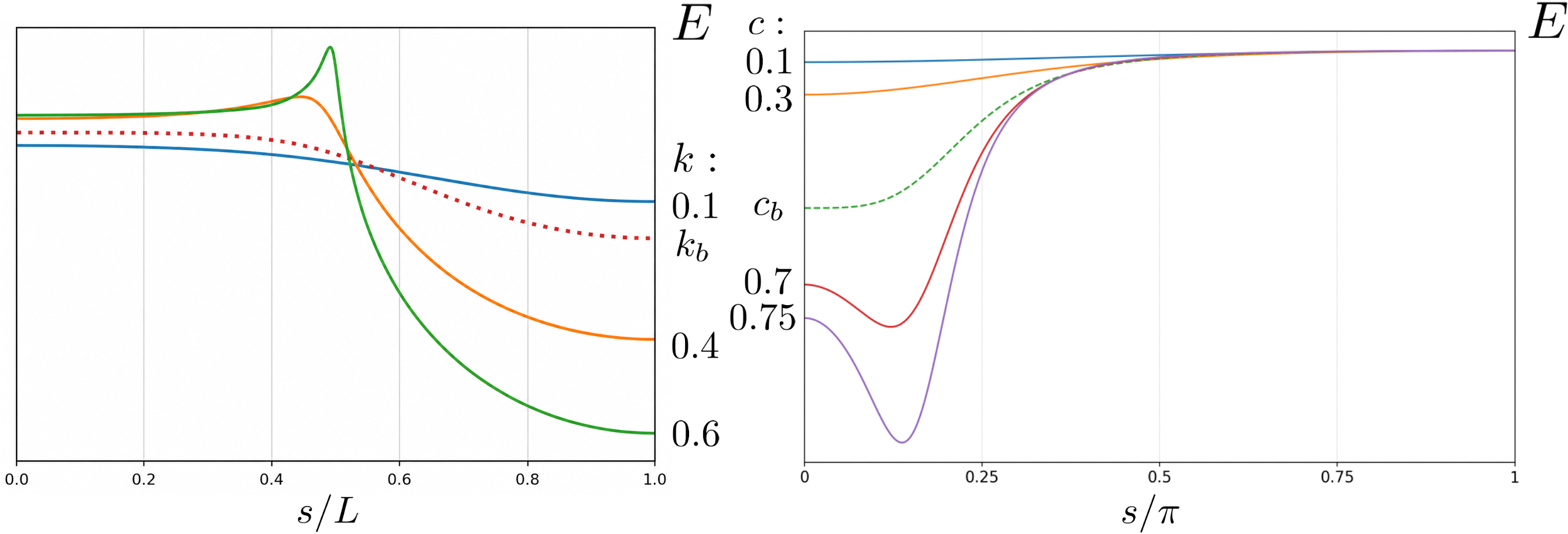}
\caption{
Left: plot of \(E(s_0;0.7)\) for the dumbbell-shaped domain considered in Section~\ref{subsec:dumbbell}. Here $k_b(0.7)\sim 0.181$.
Right: plot of \(E(s_0;0.5)\) for the perforated disk considered in Section~\ref{subsec:hole}.
Here, \(r=0.1\) is fixed, and in this case \(c_b(0.1,0.5)\sim0.58\).
}\label{fig:E-plot}
\end{figure}

\section{Comparison between the theoretical predictions and numerical simulations}\label{app:comparison}

Figure~\ref{fig:comparison} compares the trajectories of the pulse centers predicted by numerical integration of the reduced ODE~\eqref{eq:ODE2} with those obtained from direct numerical simulations of \eqref{eq:model}. The blue solid curves represent the direct numerical simulations, whereas the red dashed curves represent the predictions of the reduced ODE. The black dotted lines indicate the theoretically predicted equilibrium positions. To integrate the ODE, we evaluated the potential functions given in Propositions~\ref{prop:E-for-f-rational} and~\ref{prop:E-perforated-disk} and employed the classical fourth-order Runge--Kutta method.

For the dumbbell-shaped domain, the blue solid and red dashed curves are in excellent agreement throughout the time interval considered. This demonstrates that the reduced ODE~\eqref{eq:ODE2} accurately predicts the motion of the pulse center observed in the direct numerical simulations.

For the perforated disk, a discrepancy is observed between the equilibrium positions predicted by the reduced ODE and obtained from the direct numerical simulations. The numerical solution approaches \(s_{\mathrm{num}}\approx0.621\), whereas the reduced ODE predicts \(s_{\mathrm{theory}}\approx0.574\). Thus,
\[
  \lvert s_{\mathrm{num}}-s_{\mathrm{theory}}\rvert
  \approx 0.047
  \approx 1.47\eps
  = O(\eps).
\]
This discrepancy is consistent with the expected \(O(\eps)\) accuracy of the leading-order asymptotic approximation. Moreover, the discrepancy remains essentially unchanged under mesh refinement, indicating that it is not primarily caused by the spatial discretization.

\begin{figure}[t]
\centering
\includegraphics[width=1.0\textwidth]{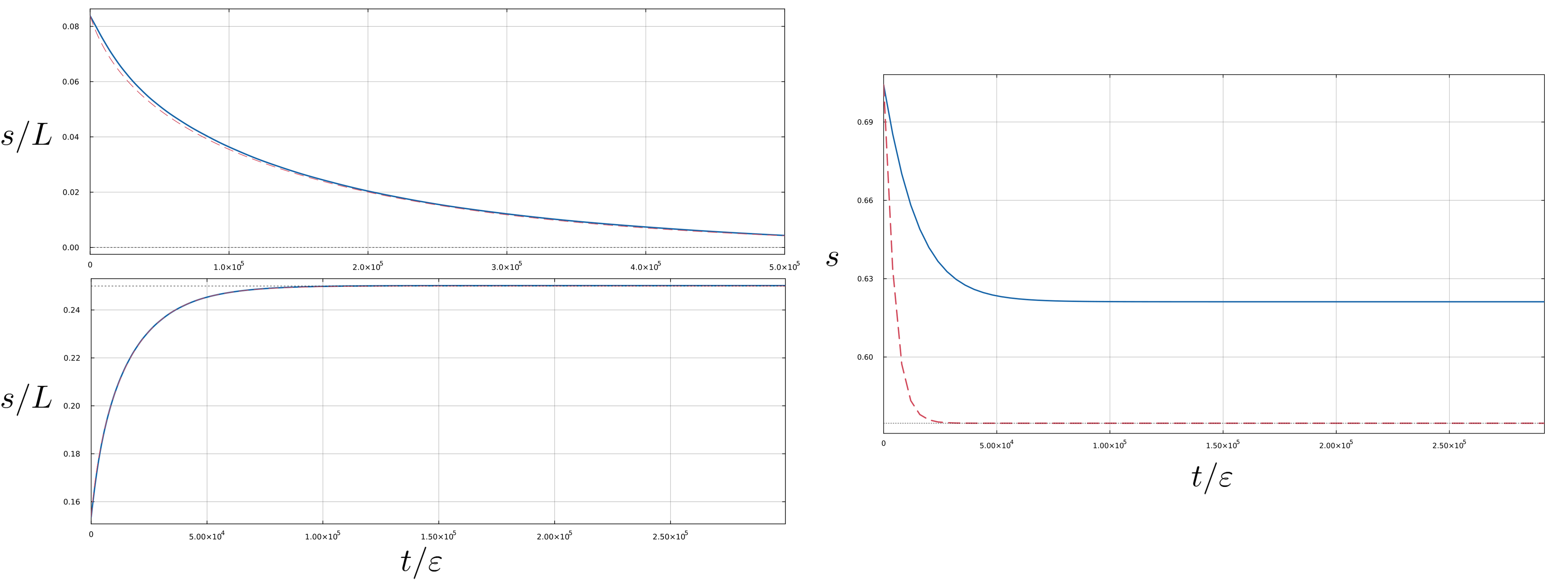}
\caption{
Left: The dumbbell case. The settings are same as Figure~\ref{fig:sim-dumbbell}. Here $L\sim5.36.$
Right: The perforated disk case. The settings are same as Figure~\ref{fig:sim-hole}
.}\label{fig:comparison}
\end{figure}

\section{Numerical method}\label{app:numerical}
 
We used the numerical scheme in \cite{cusseddu2019} for numerically solving \eqref{eq:model}, which is based on the bulk--surface finite element method proposed in \cite{madzvamuse2016}. 

We briefly describe the numerical method. The bulk–surface finite element method was employed for the spatial discretization. The bulk domain \(\Omega\) and its boundary \(\Gamma\) were discretized by bulk and surface meshes, respectively. The bulk variable \(v\) and the surface variable \(u\) were approximated using continuous piecewise-linear finite element functions on their respective meshes.
First, a predictor for the surface variable was computed using an IMEX method, in which the diffusion term was treated implicitly and the reaction term explicitly. The predicted surface solution was then used to update the bulk variable using the Crank–Nicolson method. Finally, the surface variable was corrected using the Crank–Nicolson method. This scheme preserves the discrete total mass.

The simulations were implemented in \texttt{Julia} using the finite element library \texttt{Gridap.jl}\cite{badia2020,verdugo2022}. The computational meshes were generated by \texttt{GridapGmsh}.
The code used for numerical simulations are available at \\ \url{https://github.com/RikuWatanabe-git/Geometry-induced-pulse-dynamics}.

For the nonlinear reaction term, we used
\[
f(u,v)=\Bigl(k_0+\gamma_0\frac{u^2}{1+u^2}\Bigr)v-u,
\qquad
\gamma_0>8k_0>0.
\]
In all numerical simulations, the model parameters were set to
\[
\eps^2=0.001,\qquad D=1,\qquad k_0=0.05,\qquad \gamma_0=0.79.
\]

\section{Proof of Corollary~\ref{cor:Poisson}}\label{app:poisson}
We prove Theorem~\ref{cor:Poisson} using Proposition~\ref{prop:Conformal-rep-of-E}.
Here, $\Omega$ is assumed to be a simply connected domain with smooth boundary.
We use the same notation as in Proposition~\ref{prop:Conformal-rep-of-E}.
By \cite[Example~2.12]{kozdron2005}, the excursion Poisson kernel on the unit disk satisfies
\begin{align}\label{eq:poisson-on-disk}
    P_{\partial D}(e^{i\theta_1},e^{i\theta_2})
    =
    \frac1\pi\frac1{|e^{i\theta_1}-e^{i\theta_2}|^2}
    =
    \frac{1}{4\pi\sin^2\frac{\theta_2-\theta_1}{2}}.
\end{align}
Moreover, by \cite[Proposition~2.11]{kozdron2005}, the excursion Poisson kernel is conformally invariant, and hence
\begin{align*}
    P_{\partial D}(e^{i\theta_1},e^{i\theta_2})
    &=
    |f'(e^{i\theta_1})|
    |f'(e^{i\theta_2})|
    P_{\partial\Omega}
    \bigl(
        f(e^{i\theta_1}),
        f(e^{i\theta_2})
    \bigr)
    \\
    &=
    \rho(\theta_1)\rho(\theta_2)
    P_{\partial\Omega}(x_1,x_2).
\end{align*}
Therefore,
\begin{align}
    P_{\partial\Omega}(x_1,x_2)
    =
    \frac{1}
    {4\pi
    \sin^2\frac{\theta_2-\theta_1}{2}
    \rho(\theta_1)\rho(\theta_2)}.
\end{align}
By Proposition~\ref{prop:Conformal-rep-of-E}, the right-hand side is equal to
$\pi^{-1}e^{\pi E(s_0;w_*)}$.
This proves \eqref{eq:E-Poisson}.

\section{Proof of Theorem~\ref{thm:dumbbell}}\label{app:proof-of-dumbbell}
Set $\mu:=(1-k)^2/4k$.
Then
\begin{align}\label{eq:rho-eq}
\rho(\theta)=\sqrt{\mu}\,\frac{\sqrt{\mu+\cos^2\theta}}{\mu+\sin^2\theta}.
\end{align}
In particular,
\[
\rho(-\theta)=\rho(\theta),\qquad
\rho(\pi-\theta)=\rho(\theta),\qquad
\rho(\theta+\pi)=\rho(\theta),
\qquad(\theta\in[0,2\pi)).
\]
For each \(a\in[0,2\pi)\), let \(\delta=\delta(a)\in(0,\pi)\) be the solution of
\[
\int_{a-\delta(a)}^{a+\delta(a)}\rho(\theta)\,d\theta=2w.
\]
Then we can write
\[
\theta_1=a-\delta(a),\qquad
\theta_2=a+\delta(a).
\]
The corresponding coordinate $s$ in the arc-length variable is given by
\[
s(a)=\frac{S(a-\delta(a))+S(a+\delta(a))}{2}.
\]
Hence
\[
\widetilde E(a)
:=
E(s(a))
=
-\frac1\pi\log\!\Bigl(
4\,\rho(a-\delta(a))\rho(a+\delta(a))\sin^2\delta(a)
\Bigr).
\]
Moreover,
\[
\frac{ds}{da}
=
\frac{2\,\rho(a-\delta(a))\rho(a+\delta(a))}
{\rho(a-\delta(a))+\rho(a+\delta(a))}
>0.
\]
Therefore, the local maxima and local minima of \(\widetilde E\) with respect to \(a\)
correspond one-to-one to those of \(E\) with respect to \(s\).
Also, by Remark~\ref{rem:symmetry-E}, we know that
$s=0, L/4$ are critical points of \(E\).
Let $\eta:=\delta(\pi/2)\in(0,\pi/2)$.
Since
$\rho(\pi/2+\theta)=\rho(\pi/2-\theta)$,
we have
\[
\delta'\Bigl(\frac{\pi}{2}\Bigr)=0,
\qquad
\delta''\Bigl(\frac{\pi}{2}\Bigr)
=
\frac{\rho'(\frac{\pi}{2}-\eta)}
{\rho(\frac{\pi}{2}-\eta)}.
\]
Using these identities and differentiating \(\widetilde E\) twice, we obtain
\[
\pi\,\widetilde E''\Bigl(\frac{\pi}{2}\Bigr)
=
-2\frac{\rho''(\frac{\pi}{2}-\eta)}{\rho(\frac{\pi}{2}-\eta)}
+4\left(
\frac{\rho'(\frac{\pi}{2}-\eta)}{\rho(\frac{\pi}{2}-\eta)}
\right)^2
-2\tan\eta\,
\frac{\rho'(\frac{\pi}{2}-\eta)}{\rho(\frac{\pi}{2}-\eta)}.
\]
Substituting \eqref{eq:rho-eq} and simplifying, we obtain
\[
\pi\,\widetilde E''\Bigl(\frac{\pi}{2}\Bigr)
=
\frac{
2\sin^2 \eta
}{
(\mu+\cos^2\eta)(\mu+\sin^2\eta)^2
}
\Bigl(
3\mu^2+3\mu\cos^2\eta+5\mu+\cos^2\eta+2
\Bigr)>0.
\]
Thus, \(s=L/4\) is a strict local minimum point for all \(0< k<1\).
Next, set $\theta_*:=\delta(0)\in(0,\pi)$.
Then
\[
w=\int_0^{\theta_*}\rho(\theta)\,d\theta.
\]
By the same calculation as above, we obtain
\[
\pi\,\widetilde E''(0)
=
-2\frac{\rho''(\theta_*)}{\rho(\theta_*)}
+4\left(
\frac{\rho'(\theta_*)}{\rho(\theta_*)}
\right)^2
+2\cot\theta_*\,
\frac{\rho'(\theta_*)}{\rho(\theta_*)}.
\]
Substituting the explicit expression for \(\rho\) and simplifying, we get
\[
\pi\,\widetilde E''(0)
=
\frac{
2(3\mu+2)\sin^2\theta_*
}{
(\mu+\sin^2\theta_*)(\mu+\cos^2\theta_*)^2
}
\Bigl(
-\sin^2\theta_*+g(\mu)
\Bigr).
\]
Therefore, the sign of \(\widetilde E''(0)\) coincides with the sign of
$-\sin^2\theta_*+g(\mu)$.
A direct calculation shows that
\[
g(\mu)\le0
\quad\text{if and only if}\quad
k\le k_*.
\]
This gives the classification stated in Theorem~\ref{thm:dumbbell}. The limiting behavior of \(w_b(k)\) follows by direct calculation.

\section{Proof of Theorem~\ref{thm:hole}}\label{app:proof-of-hole}
Set
\[
D(s)
:=
\bigl(1-2b\cos(s-w)+b^2\bigr)
\bigl(1-2b\cos(s+w)+b^2\bigr).
\]
By Proposition~\ref{prop:E-perforated-disk}, we have
\[
E(s)
=
C_w
-\frac{2}{\pi}\sum_{m=1}^\infty
\log\!\left(
1+\frac{4a^{2m}(1-b^2)^2\sin^2 w}{(1-a^{2m})^2D(s)}
\right).
\]
Hence, \(E\) is a strictly increasing function of \(D(s)\).

We next rewrite \(D(s)\). Set
\[
b_0:=\frac{1+b^2}{2b}\cos w.
\]
Then
\begin{align*}
D(s)
&=
\bigl(1+b^2-2b\cos(s-w)\bigr)
\bigl(1+b^2-2b\cos(s+w)\bigr)\\
&=
\bigl(2b\cos s-(1+b^2)\cos w\bigr)^2
+(1-b^2)^2\sin^2 w\\
&=
4b^2(\cos s-b_0)^2+(1-b^2)^2\sin^2w.
\end{align*}
Therefore,
\[
D'(s)
=
8b^2\sin s\,\bigl(b_0-\cos s\bigr).
\]
Since \(E\) is a strictly increasing function of \(D(s)\), the critical points and their stability are determined by \(D'(s)\). The classification in Theorem~\ref{thm:hole} then follows from the above formula.
The monotonicity and the limiting behavior of
$\arccos\frac{2b}{1+b^2}$
with respect to \(c\) are obtained by direct calculation.

\bibliographystyle{plainnat}
\bibliography{reference}

@article{alfaro2008,
  title = {The Singular Limit of the {{Allen}}--{{Cahn}} Equation and the {{FitzHugh}}--{{Nagumo}} System},
  author = {Alfaro, Matthieu and Hilhorst, Danielle and Matano, Hiroshi},
  year = 2008,
  month = jul,
  journal = {Journal of Differential Equations},
  volume = {245},
  number = {2},
  pages = {505--565},
  doi = {10.1016/j.jde.2008.01.014},
  langid = {english}
}

@article{badia2020,
  title = {Gridap: {{An}} Extensible {{Finite Element}} Toolbox in {{Julia}}},
  shorttitle = {Gridap},
  author = {Badia, Santiago and Verdugo, Francesc},
  year = 2020,
  month = aug,
  journal = {Journal of Open Source Software},
  volume = {5},
  number = {52},
  pages = {2520},
  doi = {10.21105/joss.02520},
  langid = {english}
}

@article{borgqvist2021,
  title = {Cell Polarisation in a Bulk-Surface Model Can Be Driven by Both Classic and Non-Classic {{Turing}} Instability},
  author = {Borgqvist, Johannes and Malik, Adam and Lundholm, Carl and Logg, Anders and Gerlee, Philip and Cvijovic, Marija},
  year = 2021,
  month = feb,
  journal = {npj Systems Biology and Applications},
  volume = {7},
  number = {1},
  pages = {13},
  publisher = {Nature Publishing Group},
  doi = {10.1038/s41540-021-00173-x},
  langid = {english}
}

@article{caetano2025,
  title = {Bulk-Surface Systems on Evolving Domains},
  author = {Caetano, Diogo and Elliott, Charles M. and Tang, Bao Quoc},
  year = 2025,
  month = oct,
  journal = {Journal of Evolution Equations},
  volume = {25},
  number = {4},
  pages = {103},
  doi = {10.1007/s00028-025-01130-5},
  langid = {english}
}

@misc{chakraborty2025a,
  title = {Fast Integral Methods for the {{Neumann Green}}'s Function: Applications to Capture and Signaling Problems in Two Dimensions},
  shorttitle = {Fast Integral Methods for the {{Neumann Green}}'s Function},
  author = {Chakraborty, Sanchita and Hoskins, Jeremy and Lindsay, Alan E.},
  year = 2025,
  month = nov,
  number = {arXiv:2511.08458},
  eprint = {2511.08458},
  primaryclass = {math.NA},
  publisher = {arXiv},
  doi = {10.48550/arXiv.2511.08458},
  archiveprefix = {arXiv}
}

@article{chen1992,
  title = {Generation and Propagation of Interfaces for Reaction-Diffusion Equations},
  author = {Chen, Xinfu},
  year = 1992,
  month = mar,
  journal = {Journal of Differential Equations},
  volume = {96},
  number = {1},
  pages = {116--141},
  doi = {10.1016/0022-0396(92)90146-E},
  langid = {english}
}

@article{chen2011,
  title = {The {{Stability}} and {{Dynamics}} of {{Localized Spot Patterns}} in the {{Two-Dimensional Gray}}--{{Scott Model}}},
  author = {Chen, W. and Ward, M. J.},
  year = 2011,
  month = jan,
  journal = {SIAM Journal on Applied Dynamical Systems},
  volume = {10},
  number = {2},
  pages = {582--666},
  doi = {10.1137/09077357X},
  langid = {english}
}

@article{cusseddu2019,
  title = {A Coupled Bulk-Surface Model for Cell Polarisation},
  author = {Cusseddu, D. and {Edelstein-Keshet}, L. and Mackenzie, J.A. and Portet, S. and Madzvamuse, A.},
  year = 2019,
  month = nov,
  journal = {Journal of Theoretical Biology},
  volume = {481},
  pages = {119--135},
  doi = {10.1016/j.jtbi.2018.09.008},
  langid = {english}
}

@article{cusseddu2022,
  title = {Numerical Investigations of the Bulk-Surface Wave Pinning Model},
  author = {Cusseddu, Davide and Madzvamuse, Anotida},
  year = 2022,
  month = dec,
  journal = {Mathematical Biosciences},
  volume = {354},
  pages = {108925},
  doi = {10.1016/j.mbs.2022.108925},
  langid = {english}
}

@article{diez2024,
  title = {Turing {{Pattern Formation}} in {{Reaction-Cross-Diffusion Systems}} with a {{Bilayer Geometry}}},
  author = {Diez, Antoine and Krause, Andrew L. and Maini, Philip K. and Gaffney, Eamonn A. and {Seirin-Lee}, Sungrim},
  year = 2024,
  month = feb,
  journal = {Bulletin of Mathematical Biology},
  volume = {86},
  number = {2},
  pages = {13},
  doi = {10.1007/s11538-023-01237-1},
  langid = {english}
}

@article{duda2023,
  title = {Modelling of Surface Reactions and Diffusion Mediated by Bulk Diffusion},
  author = {Duda, Fernando P. and Forte Neto, Francisco S. and Fried, Eliot},
  year = 2023,
  month = dec,
  journal = {Philosophical Transactions of the Royal Society A: Mathematical, Physical and Engineering Sciences},
  volume = {381},
  number = {2263},
  pages = {20220367},
  doi = {10.1098/rsta.2022.0367},
  langid = {english}
}

@article{giese2015,
  title = {Influence of Cell Shape, Inhomogeneities and Diffusion Barriers in Cell Polarization Models},
  author = {Giese, Wolfgang and Eigel, Martin and Westerheide, Sebastian and Engwer, Christian and Klipp, Edda},
  year = 2015,
  month = nov,
  journal = {Physical Biology},
  volume = {12},
  number = {6},
  pages = {066014},
  doi = {10.1088/1478-3975/12/6/066014},
  langid = {english}
}

@article{goehring2011,
  title = {Polarization of {{PAR Proteins}} by {{Advective Triggering}} of a {{Pattern-Forming System}}},
  author = {Goehring, Nathan W. and Trong, Philipp Khuc and Bois, Justin S. and Chowdhury, Debanjan and Nicola, Ernesto M. and Hyman, Anthony A. and Grill, Stephan W.},
  year = 2011,
  month = nov,
  journal = {Science},
  volume = {334},
  number = {6059},
  pages = {1137--1141},
  publisher = {American Association for the Advancement of Science},
  doi = {10.1126/science.1208619}
}

@article{goldstein2011,
  title = {A {{Cahn}}--{{Hilliard}} Model in a Domain with Non-Permeable Walls},
  author = {Goldstein, Gis{\`e}le Ruiz and Miranville, Alain and Schimperna, Giulio},
  year = 2011,
  month = apr,
  journal = {Physica D: Nonlinear Phenomena},
  volume = {240},
  number = {8},
  pages = {754--766},
  doi = {10.1016/j.physd.2010.12.007}
}

@article{gomez2023,
  title = {Front Propagation in the Shadow Wave-Pinning Model},
  author = {Gomez, Daniel and Lam, King-Yeung and Mori, Yoichiro},
  year = 2023,
  month = may,
  journal = {Journal of Mathematical Biology},
  volume = {86},
  number = {5},
  pages = {72},
  doi = {10.1007/s00285-023-01908-6},
  langid = {english}
}

@article{hale2005,
  title = {A {{Lyapunov-Schmidt}} Method for Transition Layers in Reaction-Diffusion Systems},
  author = {Hale, Jack K. and Sakamoto, Kunimochi},
  year = 2005,
  month = jul,
  journal = {Hiroshima Mathematical Journal},
  volume = {35},
  number = {2},
  pages = {205--249},
  publisher = {Hiroshima University, Mathematics Program},
  doi = {10.32917/hmj/1150998273},
  langid = {english}
}

@article{handy2021,
  title = {Revising {{Berg-Purcell}} for Finite Receptor Kinetics},
  author = {Handy, Gregory and Lawley, Sean D.},
  year = 2021,
  month = jun,
  journal = {Biophysical Journal},
  volume = {120},
  number = {11},
  pages = {2237--2248},
  doi = {10.1016/j.bpj.2021.03.021},
  langid = {english}
}

@article{hausberg2018,
  title = {Well-Posedness and Fast-Diffusion Limit for a Bulk--Surface Reaction--Diffusion System},
  author = {Hausberg, Stephan and R{\"o}ger, Matthias},
  year = 2018,
  month = apr,
  journal = {Nonlinear Differential Equations and Applications NoDEA},
  volume = {25},
  number = {3},
  pages = {17},
  doi = {10.1007/s00030-018-0508-8},
  langid = {english}
}

@article{ikeda2025,
  title = {Stability of Single Transition Layer in Mass-Conserving Reaction-Diffusion Systems with Bistable Nonlinearity},
  author = {Ikeda, Hideo and Kuwamura, Masataka},
  year = 2025,
  month = sep,
  journal = {Journal of Differential Equations},
  volume = {440},
  pages = {113430},
  doi = {10.1016/j.jde.2025.113430},
  langid = {english}
}

@article{ishii2026,
  title = {Spot Solutions to a Neural Field Equation on Oblate Spheroids},
  author = {Ishii, Hiroshi and Watanabe, Riku},
  year = 2026,
  month = jan,
  journal = {Communications in Nonlinear Science and Numerical Simulation},
  volume = {152},
  pages = {109172},
  doi = {10.1016/j.cnsns.2025.109172},
  langid = {english}
}

@article{Iyaniwura2021,
  title = {Optimization of the Mean First Passage Time in Near-Disk and Elliptical Domains in 2-{{D}} with Small Absorbing Traps},
  author = {Iyaniwura, Sarafa A. and Wong, Tony and Macdonald, Colin B. and Ward, Michael J.},
  year = 2021,
  journal = {SIAM Review},
  volume = {63},
  number = {3},
  pages = {525--555},
  doi = {10.1137/20M1332396}
}

@article{kagawa2024,
  title = {Critical Slowing down for Relaxation in the {{Cahn}}--{{Hilliard}} Equation with Dynamic Boundary Conditions},
  author = {Kagawa, Keiichiro and Yamazaki, Yoshihiro},
  year = 2024,
  journal = {JSIAM Letters},
  volume = {16},
  pages = {73--76},
  doi = {10.14495/jsiaml.16.73},
  langid = {english}
}

@article{kolokolnikov2003,
  title = {Reduced Wave {{Green}}'s Functions and Their Effect on the Dynamics of a Spike for the {{Gierer}}--{{Meinhardt}} Model},
  author = {Kolokolnikov, Theodore and Ward, Michael J.},
  year = 2003,
  month = oct,
  journal = {European Journal of Applied Mathematics},
  volume = {14},
  number = {5},
  pages = {513--545},
  doi = {10.1017/S0956792503005254},
  langid = {english}
}

@article{kolokolnikov2004,
  title = {Bifurcation of Spike Equilibria in the Near-Shadow {{Gierer-Meinhardt}} Model},
  author = {Kolokolnikov, Theodore and Ward, Michael J.},
  year = 2004,
  journal = {Discrete and Continuous Dynamical Systems - B},
  volume = {4},
  number = {4},
  pages = {1033--1064},
  publisher = {{Discrete and Continuous Dynamical Systems - B}},
  doi = {10.3934/dcdsb.2004.4.1033},
  langid = {english}
}

@article{kolokolnikov2005,
  title = {Optimizing the Fundamental {{Neumann}} Eigenvalue for the {{Laplacian}} in a Domain with Small Traps},
  author = {Kolokolnikov, T. and Titcombe, M. S. and Ward, M. J.},
  year = 2005,
  month = apr,
  journal = {European Journal of Applied Mathematics},
  volume = {16},
  number = {2},
  pages = {161--200},
  doi = {10.1017/S0956792505006145},
  langid = {english}
}

@article{kozdron2005,
  title = {Estimates of {{Random Walk Exit Probabilities}} and {{Application}} to {{Loop-Erased Random Walk}}},
  author = {Kozdron, Michael and Lawler, Gregory},
  year = 2005,
  month = jan,
  journal = {Electronic Journal of Probability},
  volume = {10},
  doi = {10.1214/EJP.v10-294},
  langid = {english}
}

@article{kuwamura2024,
  title = {Single Transition Layer in Mass-Conserving Reaction-Diffusion Systems with Bistable Nonlinearity},
  author = {Kuwamura, Masataka and Teramoto, Takashi and Ikeda, Hideo},
  year = 2024,
  month = nov,
  journal = {Nonlinearity},
  volume = {37},
  number = {11},
  pages = {115013},
  doi = {10.1088/1361-6544/ad7fc4},
  langid = {english}
}

@book{legall2016,
  title = {Brownian {{Motion}}, {{Martingales}}, and {{Stochastic Calculus}}},
  author = {Le Gall, Jean-Fran{\c c}ois},
  year = 2016,
  series = {Graduate {{Texts}} in {{Mathematics}}},
  volume = {274},
  publisher = {Springer International Publishing},
  address = {Cham},
  doi = {10.1007/978-3-319-31089-3},
  isbn = {978-3-319-31088-6 978-3-319-31089-3},
  langid = {english}
}

@article{macdonald2016,
  title = {A Computational Method for the Coupled Solution of Reaction--Diffusion Equations on Evolving Domains and Manifolds: {{Application}} to a Model of Cell Migration and Chemotaxis},
  shorttitle = {A Computational Method for the Coupled Solution of Reaction--Diffusion Equations on Evolving Domains and Manifolds},
  author = {MacDonald, G. and Mackenzie, J.A. and Nolan, M. and Insall, R.H.},
  year = 2016,
  month = mar,
  journal = {Journal of Computational Physics},
  volume = {309},
  pages = {207--226},
  doi = {10.1016/j.jcp.2015.12.038},
  pmcid = {PMC4896117},
  pmid = {27330221}
}

@article{madzvamuse2016,
  title = {The Bulk-Surface Finite Element Method for Reaction--Diffusion Systems on Stationary Volumes},
  author = {Madzvamuse, Anotida and Chung, Andy H.W.},
  year = 2016,
  month = jan,
  journal = {Finite Elements in Analysis and Design},
  volume = {108},
  pages = {9--21},
  doi = {10.1016/j.finel.2015.09.002},
  langid = {english}
}

@article{mangeat2019,
  title = {The Narrow Escape Problem in a Circular Domain with Radial Piecewise Constant Diffusivity},
  author = {Mangeat, Matthieu and Rieger, Heiko},
  year = 2019,
  month = oct,
  journal = {Journal of Physics A: Mathematical and Theoretical},
  volume = {52},
  number = {42},
  eprint = {1906.06975},
  primaryclass = {cond-mat},
  pages = {424002},
  doi = {10.1088/1751-8121/ab4348},
  archiveprefix = {arXiv},
  langid = {english}
}

@article{miller2023,
  title = {Generation and {{Motion}} of {{Interfaces}} in a {{Mass-Conserving Reaction-Diffusion System}}},
  author = {Miller, Pearson W. and Fortunato, Daniel and Novaga, Matteo and Shvartsman, Stanislav Y. and Muratov, Cyrill B.},
  year = 2023,
  month = sep,
  journal = {SIAM Journal on Applied Dynamical Systems},
  volume = {22},
  number = {3},
  pages = {2408--2431},
  doi = {10.1137/22M152548X},
  langid = {english}
}

@article{mori2008,
  title = {Wave-{{Pinning}} and {{Cell Polarity}} from a {{Bistable Reaction-Diffusion System}}},
  author = {Mori, Yoichiro and Jilkine, Alexandra and {Edelstein-Keshet}, Leah},
  year = 2008,
  month = may,
  journal = {Biophysical Journal},
  volume = {94},
  number = {9},
  pages = {3684--3697},
  doi = {10.1529/biophysj.107.120824}
}

@article{mori2011,
  title = {Asymptotic and {{Bifurcation Analysis}} of {{Wave-Pinning}} in a {{Reaction-Diffusion Model}} for {{Cell Polarization}}},
  author = {Mori, Yoichiro and Jilkine, Alexandra and {Edelstein-Keshet}, Leah},
  year = 2011,
  month = jan,
  journal = {SIAM Journal on Applied Mathematics},
  volume = {71},
  number = {4},
  pages = {1401--1427},
  doi = {10.1137/10079118X},
  langid = {english}
}

@article{mori2022,
  title = {Representation Formulas for Stationary Solutions of a Cell Polarization Model},
  author = {Mori, Tatsuki and Tsujikawa, Tohru and Yotsutani, Shoji},
  year = 2022,
  month = dec,
  journal = {Japan Journal of Industrial and Applied Mathematics},
  volume = {39},
  number = {3},
  pages = {1025--1053},
  doi = {10.1007/s13160-022-00537-8},
  langid = {english}
}

@article{morita2020,
  title = {Turing Type Instability in a Diffusion Model with Mass Transport on the Boundary},
  author = {Morita, Yoshihisa and Sakamoto, Kunimochi},
  year = 2020,
  journal = {Discrete and Continuous Dynamical Systems},
  volume = {40},
  number = {6},
  pages = {3813--3836},
  doi = {10.3934/dcds.2020160},
  langid = {english}
}

@article{morita2021,
  title = {Long Time Behavior and Stable Patterns in High-Dimensional Polarity Models of Asymmetric Cell Division},
  author = {Morita, Yoshihisa and {Seirin-Lee}, Sungrim},
  year = 2021,
  month = jun,
  journal = {Journal of Mathematical Biology},
  volume = {82},
  number = {7},
  pages = {66},
  doi = {10.1007/s00285-021-01619-w},
  langid = {english}
}

@article{ni1993,
  title = {Locating the Peaks of Least-Energy Solutions to a Semilinear {{Neumann}} Problem},
  author = {Ni, Wei-Ming and Takagi, Izumi},
  year = 1993,
  month = may,
  journal = {Duke Mathematical Journal},
  volume = {70},
  number = {2},
  pages = {247--281},
  publisher = {Duke University Press},
  doi = {10.1215/S0012-7094-93-07004-4},
  langid = {english}
}

@article{otsuji2007,
  title = {{A Mass Conserved Reaction--Diffusion System Captures Properties of Cell Polarity}},
  author = {Otsuji, Mikiya and Ishihara, Shuji and Co, Carl and Kaibuchi, Kozo and Mochizuki, Atsushi and Kuroda, Shinya},
  year = 2007,
  month = jun,
  journal = {PLOS Computational Biology},
  volume = {3},
  number = {6},
  pages = {e108},
  publisher = {Public Library of Science},
  doi = {10.1371/journal.pcbi.0030108},
  langid = {japanese}
}

@article{pillay2010,
  title = {An {{Asymptotic Analysis}} of the {{Mean First Passage Time}} for {{Narrow Escape Problems}}: {{Part I}}: {{Two-Dimensional Domains}}},
  shorttitle = {An {{Asymptotic Analysis}} of the {{Mean First Passage Time}} for {{Narrow Escape Problems}}},
  author = {Pillay, S. and Ward, M. J. and Peirce, A. and Kolokolnikov, T.},
  year = 2010,
  month = jan,
  journal = {Multiscale Modeling \& Simulation},
  volume = {8},
  number = {3},
  pages = {803--835},
  doi = {10.1137/090752511},
  langid = {english}
}

@article{ramirez2015,
  title = {Dendritic Spine Geometry Can Localize {{GTPase}} Signaling in Neurons},
  author = {Ramirez, Samuel A. and Raghavachari, Sridhar and Lew, Daniel J.},
  editor = {{Edelstein-Keshet}, Leah},
  year = 2015,
  month = nov,
  journal = {Molecular Biology of the Cell},
  volume = {26},
  number = {22},
  pages = {4171--4181},
  doi = {10.1091/mbc.E15-06-0405},
  langid = {english}
}

@article{ratz2014,
  title = {Symmetry Breaking in a Bulk--Surface Reaction--Diffusion Model for Signalling Networks},
  author = {R{\"a}tz, Andreas and R{\"o}ger, Matthias},
  year = 2014,
  month = aug,
  journal = {Nonlinearity},
  volume = {27},
  number = {8},
  pages = {1805--1827},
  doi = {10.1088/0951-7715/27/8/1805},
  langid = {english}
}

@article{sakajo2021,
  title = {Spot {{Dynamics}} of a {{Reaction-Diffusion System}} on the {{Surface}} of a {{Torus}}},
  author = {Sakajo, Takashi and Wang, Penghao},
  year = 2021,
  month = jan,
  journal = {SIAM Journal on Applied Dynamical Systems},
  volume = {20},
  number = {2},
  pages = {1053--1089},
  doi = {10.1137/20M1380636},
  langid = {english}
}

@article{scheel2021,
  title = {Signaling Gradients in Surface Dynamics as Basis for Planarian Regeneration},
  author = {Scheel, Arnd and Stevens, Angela and Tenbrock, Christoph},
  year = 2021,
  month = jul,
  journal = {Journal of Mathematical Biology},
  volume = {83},
  number = {1},
  pages = {6},
  doi = {10.1007/s00285-021-01627-w},
  langid = {english}
}

@article{verdugo2022,
  title = {The Software Design of {{Gridap}}: {{A Finite Element}} Package Based on the {{Julia JIT}} Compiler},
  shorttitle = {The Software Design of {{Gridap}}},
  author = {Verdugo, Francesc and Badia, Santiago},
  year = 2022,
  month = jul,
  journal = {Computer Physics Communications},
  volume = {276},
  pages = {108341},
  doi = {10.1016/j.cpc.2022.108341},
  langid = {english}
}

@article{yotsutani2015,
  title = {Global Bifurcation Sheet and Diagrams of Wave-Pinning in a Reaction-Diffusion Model for Cell Polarization},
  author = {Yotsutani, Shoji and Tsujikawa, Tohru and Nagayama, Masaharu and Kuto, Kousuke and Mori, Tatsuki},
  year = 2015,
  month = nov,
  journal = {Dynamical Systems and Differential Equations, AIMS Proceedings 2015},
  pages = {861--877},
  publisher = {American Institute of Mathematical Sciences},
  doi = {10.3934/proc.2015.0861},
  langid = {english}
}

\end{document}